\documentclass[oneside,reqno,english]{amsart}

\usepackage{geometry}
\usepackage{xifthen}
\usepackage{subfig}
\usepackage{amssymb,amsfonts,amsmath,amsthm,amscd,dsfont,mathrsfs,comment,mathtools,enumerate}
\usepackage{graphicx,float,psfrag,epsfig}

\usepackage[noadjust]{cite}

\makeatletter
\numberwithin{equation}{section}
\numberwithin{figure}{section}
\theoremstyle{plain}
\newtheorem{thm}{\protect\theoremname}[section]
\theoremstyle{plain}
\newtheorem{rem}[thm]{\protect\remarkname}
\theoremstyle{plain}
\newtheorem{cor}[thm]{\protect\corollaryname}
\theoremstyle{plain}
\newtheorem{prop}[thm]{\protect\propositionname}
\theoremstyle{remark}
\newtheorem{claim}[thm]{\protect\claimname}
\theoremstyle{plain}
\newtheorem{lem}[thm]{\protect\lemmaname}
\newtheorem{maintheorem}{Theorem}
\newtheorem{maincor}[maintheorem]{Corollary}

\newtheorem{hypothesis}{Hypothesis}

\theoremstyle{definition}{

\newtheorem*{definition*}{Definition}
}

\usepackage[colorlinks=true]{hyperref}

\usepackage[colorinlistoftodos]{todonotes}
\presetkeys{todonotes}{inline, color=green}{}

\newcommand{\bm}{{\mathbf{m}}}
\newcommand{\bu}{{\mathbf{u}}}
\newcommand{\tbm}{{\tilde{\mathbf{m}}}}
\newcommand{\balpha}{{\boldsymbol{\alpha}}}
\newcommand{\bgamma}{{\boldsymbol{\gamma}}}
\newcommand{\bbeta}{{\boldsymbol{\beta}}}
\newcommand{\btau}{{\boldsymbol{\tau}}}
\newcommand{\bzeta}{{\boldsymbol{\zeta}}}
\newcommand{\bLambda}{{\boldsymbol{\Lambda}}}
\newcommand{\bxi}{{\boldsymbol{\xi}}}
\newcommand{\bA}{{\mathbf{A}}}
\newcommand{\ba}{{\mathbf{a}}}

\newcommand{\bt}{{\mathbf{t}}}
\newcommand{\bs}{{\mathbf{s}}}
\newcommand{\bh}{{\mathbf{h}}}
\newcommand{\bM}{{\mathbf{M}}}
\newcommand{\bZ}{{\mathbf{Z}}}
\newcommand{\bC}{{\mathbf{C}}}
\newcommand{\cX}{{\mathcal{X}}}
\newcommand{\cM}{{\mathcal{M}}}

\newcommand{\one}{{\mathbf{1}}}

\newcommand{\abbr}[1]{{\sc\lowercase{#1}}}

\newcommand{\bS}{{\mathbb{S}}}
\newcommand{\bR}{{\mathbb{R}}}
\newcommand{\bP}{{\mathbb{P}}}
\newcommand{\tbP}{{\tilde{\mathbb P}}}
\newcommand{\bE}{{\mathbb{E}}}
\newcommand{\tbE}{{\tilde{\mathbb E}}}
\newcommand{\cN}{{\mathcal N}}
\newcommand{\cE}{{\mathcal E}}
\newcommand{\Cmu}{C_\mu}
\newcommand{\CA}{C_{\mathbf A}}
\newcommand{\CtA}{C_{\tilde{\mathbf A}}}

\newcommand{\eps}{\varepsilon}

\newcommand{\bJ}{{\mathbf{J}}}
\newcommand{\tbJ}{{\tilde{\mathbf{J}}}}
\newcommand{\tbA}{{\tilde{\mathbf{A}}}}
\newcommand{\tbX}{{\tilde{\mathbf{X}}}}

\newcommand{\bx}{{\mathbf x}}
\newcommand{\bX}{{\mathbf{X}}}

\newcommand{\bG}{{\mathbf{G}}}
\newcommand{\bB}{{\mathbf{B}}}
\newcommand{\cY}{{\mathcal{Y}}}

\newcommand{\cL}{{\mathcal{L}}}

\newcommand{\bsigma}{{\boldsymbol{\sigma}}}
\newcommand{\bvsigma}{{\boldsymbol{\varsigma}}}
\newcommand{\bSigma}{{\boldsymbol{\Sigma}}}

\makeatother

\providecommand{\claimname}{Claim}
\providecommand{\corollaryname}{Corollary}
\providecommand{\lemmaname}{Lemma}
\providecommand{\propositionname}{Proposition}
\providecommand{\remarkname}{Remark}
\providecommand{\theoremname}{Theorem}

\begin{document}
\title[Universality for diffusions interacting through a random matrix]{
Diffusions interacting through a random matrix:\\ 
universality via stochastic Taylor expansion}

\vspace{-.25cm}
\author{Amir Dembo}
\address{A.\ Dembo \hfill \break 
Department of Statistics and Department of Mathematics\\ Stanford University\\ Stanford, CA.}
\email{adembo@stanford.edu}

\author{Reza Gheissari}
\address{R.\ Gheissari\hfill\break
Departments of Statistics and EECS \\ University of California, Berkeley\\ Berkeley, CA.}
\email{gheissari@berkeley.edu}

\thanks{
\newline
{\bf AMS (2010) Subject Classification:}
Primary: 60J60, Secondary:  60B20, 60J35,  60K35, 82C44.
\newline
{\bf Keywords:} Stochastic differential equations, Universality, Markov semi-group, Random matrices,
Disordered systems, Langevin dynamics, Gradient flows.}

\begin{abstract}
\vspace{-.25cm}
Consider $(X_{i}(t))$ solving a system of $N$ stochastic differential
equations interacting through a random matrix $\mathbf J = (J_{ij})$ with independent (not necessarily identically distributed) random coefficients. We show that the trajectories of averaged observables of $(X_i(t))$,
initialized from some $\mu$ independent of~$\mathbf J$, are universal, i.e.,  only
depend on the choice of the distribution $\mathbf{J}$ through its
first and second moments (assuming e.g., sub-exponential tails).
We take a general combinatorial approach to proving universality for
dynamical systems with random coefficients, combining a stochastic
Taylor expansion with a moment matching-type argument. Concrete settings for which our results imply universality include aging in the spherical SK spin glass, and Langevin dynamics and gradient flows for symmetric and asymmetric Hopfield networks. 
\end{abstract}

\maketitle

\vspace{-.75cm}

\section{Introduction}
Markov processes with random coefficients arise in numerous contexts: e.g., dynamics of spin glasses, optimization on random landscapes, and learning with neural networks. In many cases, when the underlying randomness is Gaussian, they have been found to give rise to a rich class of behaviors, including metastability, trapping, and aging. In this paper, we analyze a class of stochastic differential systems (\abbr{SDS}'s) in their high dimensional limit, where the couplings are linear and encoded by a random matrix. We show that trajectories of polynomial statistics of the \abbr{SDS} are universal: they
have the same high-dimensional behavior if one replaces the Gaussian interaction matrix by a non-Gaussian one with the same mean and variance profiles. 

Universality, can broadly be described as the phenomenon that for high dimensional ensembles $(X_i)_{i\le N}$ governed by a large number of independent random variables $(Z_i)_{i\le N}$, macrocopic statistics of the ensemble only depend on the laws of $(Z_i)$ through their low moments. Of course, the most classical example of universality is the central limit theorem (\abbr{CLT}), where $(X_i)=(Z_i)$, and the statistic is the normalized sum. Slightly more involved examples are  invariance principles, where the limiting Brownian motion only depends on the distribution of the random walk increments through its first and second moments.    

Lindeberg's classical proof of the \abbr{CLT} iteratively replaces $Z_i$ with $\tilde Z_i$ (Gaussian with the same mean and variance) and shows that the cumulative effect of these replacements is microscopic. This approach has proven to be very robust, and has been generalized e.g., to polynomials $f(Z_1,\ldots,Z_N)$ in~\cite{Rotar,MOO05} and more generally, smooth functions with bounded derivatives in~\cite{ChatterjeeLindeberg,ChatterjeeInvariance}. A more combinatorial approach is a \emph{moment matching} argument to compare moments of statistics $f(X_1,\ldots,X_N)$ to moments of $f(\tilde X_1,\ldots,\tilde X_N)$ and showing that the difference is dominated by the differences in the first few moments of $Z_i$ and $\tilde Z_i$. 

With these approaches, universality has been proven in a wide range of ensembles where the relationship between $(X_i)$ and $(Z_i)$ is more complicated. A fundamental example is when $(X_i)$ are the eigenvalues of a random matrix with entries $(Z_i)$. There, the empirical distribution of $(X_i)$ is well-known to have the same limit (e.g., the semi-circle law for Wigner matrices~\cite{Wigner}). In the last decade, remarkably, universality has been found to extend to local statistics of the ensemble $(X_i)$ e.g., typical size of gaps between eigenvalues, and $k$-point correlations. Universality in random matrix theory has been a tremendous success and we cannot hope to do justice to the literature therein; we instead refer to the seminal works~\cite{ErdosShleinYauInventiones,TauVuActa} and the surveys~\cite{TauVuSurvey,ErdosYauBook}. 

Another class of ensembles for which universality has been shown is disordered interacting particle systems from statistical physics, and in particular the family of mean-field spin glass models. A canonical example of these are spin glasses where $N$ particles in states $(X_i)$, interact through a random symmetric coupling matrix (or in the case of higher order interactions, tensor) composed of independent entries $Z_i$. More precisely, with these interactions, they are endowed with an energy landscape, or Hamiltonian, that is topologically complex, and $(X_i)$ are drawn from the corresponding Gibbs distribution. The statistics of $(X_i)$ in such families of spin glasses have been found to exhibit an extremely rich and varied phase diagram featuring phenomena like breaking of ergodicity and replica symmetry~\cite{PanchSKBook}. Most of their analysis, including the calculation of the free energy, and the proof of the celebrated Parisi formula for the overlap distribution, were first carried out in the Gaussian setting~\cite{GuerraToninelli,TalagrandAoM,PanchUlt13}. Talagrand later showed that these also held in the case of Bernoulli $(Z_i)$ in~\cite{TalagrandGaussianBernoulli}; this universality was extended to general $(Z_i)$ as an application of~\cite{ChatterjeeLindeberg}.

The dynamics (Markov processes exploring the Hamiltonian) for such spin glass models are a prototype and motivating force for this paper. The general setting we consider here is that of a system of $N$ linearly coupled \abbr{SDE}'s, where the couplings are encoded in a \emph{random} matrix $\bJ$, and driven by $N$ independent Brownian motions. That is, $\bX_t  = (X_1(t),\ldots,X_N(t))$ is the solution to the \abbr{SDS}   
\begin{align}\label{eq:sds-intro}
    \begin{cases}d\bX_t = \bJ^T \bX_t dt + \boldsymbol h  dt +  {\boldsymbol \Sigma(\bX_t)} d\bB_t \\
    \bX_0 \sim \mu \in \cM_1 (\bR^N)\end{cases}\,,
\end{align}
where $\bJ$ is a random matrix with independent entries (up to, possibly, a symmetry constraint) and variance profile $\bm = (m_{ij})_{i,j}$ scaled such that $\bE[\|\bJ\|_2] =O(1)$, $\bh$ is a bounded drift vector, and 
$\bSigma$ is an affine transform of $\bX_t$. Note that for
$\bSigma(\bX_t)$ non-constant, 
we do not expect 
to have an explicit closed-form solution to~\eqref{eq:sds-intro}.

In the $N\to\infty$ limit, the diffusions of~\eqref{eq:sds-intro} encompass many interesting and well-studied models of Markov processes with random coefficients, and give rise to rich and varied behavior. This includes metastability, aging, and non-Markovian limiting evolution equations, in e.g., randomly coupled (geometric) Brownian motions, and Langevin dynamics and gradient flows for the spherical Sherrington--Kirkpatrick (\abbr{SK}) spin glass and symmetric and asymmetric Hopfield nets~\cite{Hopfield,CrisantiSompolinsky,Kinzel,Hertz-et-al,BAG1}: concrete applications are described in Section~\ref{sec:applications}. In many such examples, the analysis is more tractable when $\bJ$ is Gaussian and one can use tools like Gaussian integration by parts, Girsanov, and the rotational invariance of the Gaussian ensemble.   

In this paper, we develop a simple combinatorial framework for proving universality for the solution trajectories of \abbr{SDS}'s of the form~\eqref{eq:sds-intro}. Before describing our approach, we explain a few difficulties one encounters when trying to prove universality for solutions of randomly coupled dynamical systems, using some of the approaches described above for other universality results. We begin by considering a Lindeberg approach where we examine the effect that re-sampling one $J_{ij}$ has on an averaged statistic $F(t)  = F(X_1(t),\ldots,X_N(t))$. The obstacle in employing such an approach is that changing $J_{ij}$ 
to $\tilde J_{ij}$ on $X_j(t)$, say, beyond affecting the drift 
\[
\sum_{1\le i \le N} J_{ij} X_i(t) + h_j\,,\]
of the $j$-th coordinate of the \abbr{SDS}, 
also induces a \emph{highly non-linear} effect both on $X_j(t)$ and on $X_i(t)$ for all  $i\neq j$. The problem instead lends itself to comparing the effect of $\bJ \to \tbJ$ in a more averaged way. 

An alternative approach would be to use the linear structure of the problem in a strong way, relying on sharp universality results on the spectra of random matrices to study the problem. This approach, while feasible if $\bSigma(\bX_t)$ is constant, requires one to diagonalize the problem without loss of generality---i.e., it requires an assumption of joint rotational invariance for the laws of $(\bX_0, \bJ,\bB)$. In~\cite{BADG01}, such an approach is followed for analyzing the dynamics of the spherical \abbr{SK} model, and their results hold assuming the law of $\bJ$ is invariant under the orthogonal group, and its spectrum satisfies certain large deviation estimates satisfied by the \abbr{GOE}. However, this restriction would not include the cases of e.g., the uniform measures on $[-1,1]^N$ and $\{\pm 1\}^N$ absent the rotational symmetry, and could not include the case of non-constant $\bSigma(\bX_t)$.

Very recently,~\cite{DLZ19} proved a universality result for the dynamics of the \emph{asymmetric} Langevin dynamics for the \emph{soft-spin \abbr{SK} model}. There they used large deviations theory to obtain exponential control on the empirical measure on sample paths---as obtained in the Gaussian setting in~\cite{BAG1,BAG2}---together with sharp control via Girsanov's theorem on the Radon--Nikodym derivative between the Gaussian paths and those driven by non-Gaussian $\bJ$ on short time scales, to show universality for the empirical measure $\cL_N = \frac{1}{N}\sum_{i} \delta_{X_i(t)}$. While such an approach allows for 
a deterministic non-linearity in the drift through a (double-well) confining potential, it cannot 
handle degenerate diffusions, e.g. the gradient flow. Further, the need for 
control on the trajectories at the exponential scale forces
\cite{DLZ19} to consider only asymmetric i.i.d.\ $\bJ$ (whereby the Radon--Nikodym 
derivative is a product of functions of independent rows of $\bJ^T$).

We introduce a simple combinatorial approach to proving universality for \abbr{SDS}'s of the form of~\eqref{eq:sds-intro}, similar in flavor to the moment method. Namely, we avoid the inherent difficulty of the problem, that the transformation $\bJ \to \tbJ$ affects $X_j(t)$ through both $(J_{ij})_i \to (\tilde J_{ij})_i$ \emph{and} $(X_i(t))_i \to (\tilde X_i(t))_i$. We do so by Taylor expanding the semigroup $P_t f= \bE_{\bX_0} [f(\bX_t)]$ in powers of the infinitesimal generator: each term appearing in this expansion is a polynomial in $(x_i),(J_{ij})$ evaluated at $\bX_0$ where, crucially, the initial data is independent of $J_{ij}$. One then finds that on order one timescales, the predominant contribution to $\bE[P_t f]$ is from polynomials whose degree in $(J_{ij})_{i,j}$ is at most two. We refer to Section~\ref{sec:proof-strategy} for more details.

This approach works quite generally, and is robust to symmetric and asymmetric choices of $\bJ$ with non-homogenous means and variances, and general choices of diffusion coefficients in~\eqref{eq:sds-intro}, including $\bSigma(\bX_t)$ non-constant making the diffusion non-linear, and $\bSigma \equiv 0$ corresponding to a deterministic dynamical system. Lastly, the analysis works for arbitrary initialization independent of $\bJ$.  
The assumption of linear drift is, of course, important, and one would like to be able to drop it. We emphasize, though, that this is primarily used in order to justify the absolute convergence of the Taylor expansion of the semigroup, which one could hope to justify by other means for higher order diffusions given that a strong solution exists; the remaining combinatorial framework for moments of the generator may then generalize. We discuss this in Remark~\ref{rem:p-spin}.

We end this section by mentioning two recent results~\cite{BLM15,ChenLamUniversality} showing universality for a Lipschitz family of approximate message passing (\abbr{AMP}) algorithms---a discrete-time state evolution that has found many applications to inference and optimization in high dimensions. Some of the ideas there appear similar in spirit to our approach, using a combinatorial approach to control moments of the final state of the \abbr{AMP}. 
All the same, the general setting of~\eqref{eq:sds-intro} introduces many key differences e.g., the diffusions of~\eqref{eq:sds-intro} are in general non-linear, not globally Lipschitz, and have a built-in stochasticity.

\subsection{Setup: diffusions with random linear interactions}
Consider an $N$-dimensional stochastic differential system
with a mixture of random and deterministic linear interactions, along with possibly, some constant drifts. More precisely,
consider the \abbr{SDS} $\bX_t := (X_{i}(t))_{i=1}^{N}$ driven by the following parameters. 

Suppose that for some matrix $\bm = (m_{ij})_{i,j}$ we have random interactions
given by the random matrix  
\begin{align*}
\bA = (A_{ij})_{1\leq i,j\leq N}\,,\qquad\text{where } \qquad
& \bE[A_{ij}]=0, \quad\bE [A_{ij}^{2}]=m_{ij} \,.
\end{align*}
We assume that the entries $A_{ij}$ are either fully independent, or are independent up to a symmetry constraint $A_{ij} = A_{ji}$. Let $\bP_{\bA}$ be the law of $\bA$.
In order to scale the interactions to have an order one cumulative effect, it will be convenient to work with the rescaled interactions matrix $\bJ$ given by 
\[
\bJ := N^{-1/2} \bA\,.
\]
We then denote the distribution induced by $\bP_\bA$ on $\bJ$ by $\bP_{\bJ}$.

We further consider
additional deterministic interactions satisfying, for some constant $C_{\bLambda}<\infty$, 
\begin{align*}
\bLambda = (\Lambda_{ij})_{1\leq i,j\leq N}\,,\quad\text{where } \quad & 
\max_i \| (\Lambda_{ij})_j \|_1 \le C_{\bLambda} \;\; \mbox{ and } \;\;
\sup_{i,j}|\Lambda_{ij}|
\le \frac{C_{\bLambda}}{\cN_{\bLambda}}\;\; \mbox{ for } \;\; \cN_{\bLambda}:= \max_j \|(\Lambda_{ij})_i\|_{0}
\end{align*}
(the $\|\cdot \|_0$-norm of a vector is its number of non-zero entries).
We also consider external drift parameters 
\begin{align*}
\bh = (h_{i})_{1\leq i\leq N}\,,\qquad\text{where } \quad & \sup_{i \leq N}|h_{i}| \le C_{\bh}
\quad \mbox{ for a constant } \;\; C_{\bh}<\infty\,,
\end{align*}
and diffusion coefficients $\bSigma (\bX_t)$ governed by the 
matrix
$$
\bsigma = (\sigma_{ij})_{0\le i \le N, 1\le j \le N} \;\; \mbox{where } \;\; 
\sup_{1\le j\le N} |\sigma_{0j}| \le C_{\bsigma}  \quad \mbox{and} \;\;
\sup_{1\le i,j\le N} |\sigma_{ij}| \le \frac{C_{\bsigma}}{\cN_{\bsigma}}\quad \mbox{for} \quad\cN_{\bsigma} : = \max_j \|(\sigma_{ij})_i\|_0 \,.
$$

The \abbr{SDS} $(\bX_t)_{t\geq 0}=(X_{1}(t),X_{2}(t),\ldots,X_{N}(t))_{t\geq 0}$ 
initialized from some random $\bX_0$ distributed according to a product measure $\mu$ 
is driven by a standard Brownian motion $\bB_t = (B_1(t),\ldots,B_N(t))$ as follows
\begin{equation}\label{eq:sde-def}
dX_{j}(t)=  \sum_{i=1}^N J_{ij}X_{i}(t) dt +\sum_{i=1}^N \Lambda_{ij}X_{i}(t) dt +h_{j}dt+
 \sqrt{2} \Big( \sum_{i=0}^N \sigma_{ij} X_i(t)\Big)
dB_{j}(t) \,,
\qquad X_i(0)\sim \mu_i \,,
\end{equation}
where for ease of notation, we hereon set $X_0(t) \equiv 1$ so that $(\sigma_{0j})_{j\ge 1}$ capture the constant diffusion coefficients. We denote the martingale part of $\bX_t$ by  
\begin{align}\label{eq:martingale-def}
\bM_t = (M_j(t))_{j\le N}\,, \qquad\mbox{where}\qquad dM_j(t) = 
\sqrt{2} \Big( \sum_{i=0}^N \sigma_{ij} X_i(t)\Big)
dB_j(t)\,.
\end{align}
The process $\bX_t$ is well-defined for a.e.\ $\bJ$ and all $t \geq 0$ (as we have finite, 
possibly $N$-dependent operator norms $\|\bJ\|_{2}$, $\|\bLambda\|_2$ and $\|(\sigma_{ij})_{i \ge 1}\|_2$, see e.g., \cite[Theorem 5.2.1]{Oksendal}). 

\medskip
\noindent \textbf{Notational comment.} There are three distinct sources of randomness above dictating the law of the solution $\bX_t$ to~\eqref{eq:sde-def}: the law of the interaction matrix $\bP_{\bJ}$, the law of the Brownian motions, denoted $\bP_{\bB}$, and the law of the initial data $\mu$---each of these are product measures and we do not distinguish notationally between the law of the individual entries of $\bJ,\bB$ or $\bX_0$ and the ensembles. 

In proving universality, we consider the difference between $\bP_{\bJ}, \bP_{\tbJ}$ induced by two different distributions $\bP_{\bA}$ and $\bP_{\tbA}$ over mean-zero random matrices $\bA, \tbA$ with independent entries (possibly up to symmetry), having 
matching variance profiles $\bm = \tbm$. For ease of notation, we will henceforth use  
\[
\bP = \mu \otimes \bP_\bJ \otimes \bP_\bB\,, \qquad \mbox{and}\qquad \tbP = \mu \otimes \bP_{\tbJ} \otimes \bP_\bB\,,
\]
and denote the corresponding expectations $\bE$ and $\tbE$ respectively.

\subsection{Main results}
We begin by describing the observables to which our universality results apply. The building blocks of these observables are chosen among the family of vector valued functions,
\begin{align}\label{eq:building-blocks}
 {\mathfrak F} = \begin{cases} \one_t =(1,\ldots,1) \\  \bX_t = (X_1(t) ,\ldots,X_N(t)) \\  \bG_t = (G_1(\bX_t),\ldots,G_N(\bX_t)) \\ \bM_t = (M_1(t),\ldots,M_N(t))  \end{cases}, \qquad \mbox{where} \qquad G_j(\bx) = \sum_{i=1}^N J_{ij} x_i\,.
\end{align}
We establish universality in the mean for weighted empirical averages of monomials 
in functions from ${\mathfrak F}$ evaluated at 
a finite collection of times.  Specifically, fixing an $m$-tensor $\ba = (a_{i_1,\ldots,i_m})$ with entries bounded by $C_\ba$ and a $p$-tuple of times $\bt = (t_1,\ldots,t_p)$, for every $\ell \le m$, fix $p$ observables $\cY^{(\ell,1)},\ldots,\cY^{(\ell,p)} \in {\mathfrak F}$ which are to be evaluated at these $p$ times. That is,
\begin{align}\label{eq:F-general-form}
F (\bt) = \frac{1}{N^m} \sum_{i_1,\ldots,i_m \le N}  a_{i_1,\ldots,i_m} F^{(1)}_{i_1}(\bt) \cdots F^{(m)}_{i_m}(\bt)\,, \qquad \mbox{where}\qquad F^{(\ell)}_i(\bt) =  \cY^{(\ell,1)}_{i} (t_1) \cdots \cY^{(\ell, p)}_{i} (t_p)\,.
\end{align}
\medskip
We also need to add a sub-exponential tail constraint on $\mu$ and $\bP_\bA$ beyond the minimal assumptions of zero-mean and matching variances of $\bP_{\bA}$ and $\bP_{\tbA}$; this is henceforth referred to as Hypothesis~\ref{hyp:1}.

\begin{hypothesis}\label{hyp:1}
Assume that the law $\mu$ is a product of $\mu_i$ of $X_i(0)$ having finite moments of all order, which are  bounded
uniformly over $i$ and $N$. That is, there exist $\Cmu(r) \ge 1$ such that for any $r$ finite,
\begin{equation}\label{def:Cmu}
\sup_{N} \sup_{i\le N}  \bE [|X_i(0)|^r]  \le \Cmu(r) \,.
\end{equation}
Further assume $\bP_\bA$ has uniformly bounded 
exponential tails, i.e., the following equivalent properties hold: 
\begin{align}\label{eq:sub-exp-mom}
 \sup_{N}  \sup_{i,j\le N} \bE [ e^{\eps |A_{ij}|}] &  < \infty\,, \qquad \qquad \qquad \quad \quad \quad \,\,\,
 \qquad \mbox{for some } \;\; \varepsilon>0 \,, \\
 \sup_{N}  \sup_{i,j\le N} \bE[  |A_{ij}|^{\ell} ] &  \le \,  (\ell-1)! \, \CA^{\ell/2}  \,, \qquad \forall \ell \ge 1\quad \mbox{and some } \;\; C_\bA<\infty \,.
\label{eq:unif-mom-bd}
\end{align}
\end{hypothesis}

For ease of notation for dependencies on constants, we denote by 
$\bC_\star :=  \max\{\CA^{1/2},\CtA^{1/2},C_\bLambda,C_\bh, C_\bsigma^2
\}$ (where $C_{\tbA}$ is the constant $C_\bA$ with respect to distribution $\bP_{\tbA}$), and state our first result, on universality at the level of the mean 
(hence also of moments), for observables~\eqref{eq:F-general-form}. 

\begin{maintheorem}\label{thm:main-expectations}
 Let $\mu, \bP_\bA, \bP_{\tbA}$ satisfy Hypothesis~\ref{hyp:1} and suppose that 
 $\bA, \tbA$, symmetric or independent, are mean-zero of 
 matching variance profile $\bm= (m_{ij})_{i,j}$. For any $T,m,p<\infty$ and 
$\ba \in \bR^{N^m}$ with $\|\ba\|_\infty \le C_{\ba}$, there exists 
$C (T,m,p,C_\ba, \bC_\star,\Cmu)<\infty$, such that for every $N$ and
$F$ as in~\eqref{eq:F-general-form} with $(\cY^{(\ell,1)},\ldots, \cY^{(\ell, p)}) \in {\mathfrak F}$,
\begin{align*}
\sup_{\bt \in [0,T]^p} \big|\bE [ F(\bt)] - \tbE [ F(\bt)] \big|  \le C N^{-1/2}\,.
\end{align*}
In particular, 
$\big| \bE [ F(\bt)] - \tbE [ F(\bt)] \big|  \to 0$ as $N \to \infty$, uniformly in 
$\bt \in [0,T]^p$.
\end{maintheorem}

Theorem~\ref{thm:main-expectations} follows from a more general result bounding the difference in expectations for each individual monomial $F_i^{(\ell)}$ from~\eqref{eq:F-general-form} with $(\cY^{(\ell,1)},\ldots, \cY^{(\ell, p)}) \in {\mathfrak F}$. As a special case,
see Proposition~\ref{prop:difference-in-expectations},
 we find that the moments of each 
 spin $X_i(t)$ are universal. Specifically, for every fixed $k$,
\begin{equation}
\sup_{t\in [0,T]} \max_{1\le i \le N} \big| \bE[X_i(t)^k] - \tbE[X_i(t)^k] \big|  = O(N^{-1/2})\,.
\end{equation}

For a more restricted class of observables, with additional restrictions on the distributions $\mu$ and $\bP_{\bA}$ and $\bP_{\tbA}$, we extend the above to almost sure and $L^q$ convergence for the observable trajectories. Precisely, we restrict the observables of~\eqref{eq:F-general-form} to 
$m=1$ and $p=2$, leaving, the following quadratic observables
\begin{align}\label{eq:possible-observables}
F(\bt)  = F_{\cY, \cY',\ba} ( \bX_t, \bX_{t'}) := \frac 1{N}  \sum_{i=1}^N a_i \cY_i(t) \cY_i' (t') \,.
\end{align}
In order to extend Theorem~\ref{thm:main-expectations} to a convergence for the trajectories of these observables, we further need to assume that $\bSigma$ is constant, so that $\bM_t$ is just a scaled Brownian motion, and assume the following concentration property on $\mu, \bP_\bA, \bP_{\tbA}$, which we refer to as Hypothesis~\ref{hyp:2}.  

\begin{hypothesis}\label{hyp:2} 
A sequence of probability measures $(\bP^{(n)})_{n\ge 1}$ over $\bZ_n$ in metric spaces 
$(\cX_n,d)$ satisfies \emph{exponential concentration for Lipschitz functions} if there exists some $C>0$ such that for any sequence of $1$-Lipschitz functions $f_n: (\cX_n, d) \to (\bR, |\cdot|)$ and all $\lambda>0$, 
\begin{align}\label{eq:exp-conc}
    \bP^{(n)} \big(|f_n(\bZ_n) - \bE[f_n(\bZ_n)]|>\lambda\big)\leq C\exp(-\lambda/C)\,.
\end{align}
Assume that $\mu,\bP_{\bA}$ respectively satisfy exponential concentration for Lipschitz functions on $\bR^N$ and  $\bR^{N^2}$ (or $\bR^{N(N+1)/2}$ if $\bA$ is symmetric), equipped with their Euclidian norms, for some $C_\mu,C_\bA>0$.
\end{hypothesis}
\begin{rem} Recall, 
from the theory of measure concentration, that Hypothesis~\ref{hyp:2} holds for any distribution on $\bR^n$
which satisfy a Poincar\'e inequality with constant $c>0$ (independent of $n$), namely  for all nice $f$ one 
has that
$\mbox{Var} [ f(\bZ_n) ] \le c \bE[ |\nabla f(\bZ_n)|^2]$ (see~\cite{GromovMilman}). 
By the tensorization of the Poincar\'e inequality, if $\bZ_n = (Z_1,\ldots,Z_n)$, and each of the laws
of $Z_i$ satisfy this inequality, then the product also satisfies it with the worst constant $c$. 
Having here product measures $\mu, \bP_\bA$, the marginal laws can come from any 
distribution satisfying a Poincar\'e inequality in $n=1$. These include (see e.g.,~\cite{Vershynin})
\begin{itemize}
	\item Exponential, Gaussian, and log-concave measures of the form $\exp(- V(x))$ for $V(x)$ strictly convex,
   	 \item Linear functionals of r.v.'s having a Poincar\'e inequality: e.g., the uniform measure on $[-1,1]$.
\end{itemize}
\end{rem}

The next theorem shows that under Hypothesis \ref{hyp:2},
any $F$ of the form \eqref{eq:possible-observables} concentrates around its mean.
\begin{maintheorem}\label{thm:conc}
Suppose $\mu$, $\bP_\bA$ satisfy Hypotheses~\ref{hyp:1}--\ref{hyp:2} and 
the diffusion coefficients have $\sigma_{ij}=0$ if $i\ne 0$. Then, for some 
$C(T,C_\ba, \bC_\star, \Cmu)>0$, any  
$\|\ba\|_\infty \le C_\ba$, 
every $F$ as in~\eqref{eq:possible-observables} with $\cY,\cY'\in {\mathfrak F}$, all $\lambda>0$ 
and $N \ge N_0(T,C_\ba,\bC_\star,\Cmu)$,
\begin{align}\label{eq:wts-trajectory-tail-bound}
\bP  \Big( \sup_{\bt\in [0,T]^2} |F(\bt) - \bE [F(\bt)]| 
\ge  \lambda \Big) \leq p_N(\lambda) :=
 {\begin{cases} N^C e^{ - \lambda \sqrt N/C} \,, & \lambda\leq C \\ 
 e^{- (\log \lambda) \sqrt{ N}/C}\,, & \lambda>C \end{cases}}\,.
\end{align}
\end{maintheorem}

(One might observe that the $\exp( - \Omega(\sqrt N))$ concentration in~\eqref{eq:wts-trajectory-tail-bound} differs from the more traditional $\exp( - \Omega(N))$ concentration in e.g.~\cite{BADG01,BADG06}; such differences, which recur throughout the paper, are because our Hypothesis~\ref{hyp:2} allows for merely sub-exponential, as opposed to Gaussian, tails.)
Combining Theorems~\ref{thm:main-expectations} and~\ref{thm:conc} we get  
the following strong universality for such quadratic observables.
\begin{maincor}\label{cor:main}
Suppose $\mu, \bP_\bA, \bP_{\tbA}$ satisfy Hypotheses~\ref{hyp:1}--\ref{hyp:2}, where $\bA, \tbA$, symmetric or independent, are mean-zero and have matching variance profile $\bm= (m_{ij})_{i,j}$. Let $F(\cdot)$ and $\tilde F(\cdot)$ be as in \eqref{eq:possible-observables}, for 
$\ba\in \bR^{N}$ such that $\|\ba\|_\infty \le C_\ba$, with respect to the corresponding solutions 
$\bX_t$, $\tbX_t$ for~\eqref{eq:sde-def} with constant $\bSigma$, i.e., $\sigma_{ij}=0$ if $i \ne 0$.
Then, for every $T<\infty$ we have that as $N \to \infty$, 
\begin{align*}
  Z_N :=  \sup_{\bt\in [0,T]^2} \big| F(\bt) - \tilde{F}(\bt)\big| \to 0\qquad \mbox{almost surely, and in $L^q$ for $q\ge 1$}
    \,.
\end{align*}
\end{maincor}
\begin{proof} The observables of~\eqref{eq:possible-observables} correspond to the $m = 1$ and $p=2$ 
case of~\eqref{eq:F-general-form}, so Theorem~\ref{thm:main-expectations} applies here with 
some constant $C_1=C(T,m,p,C_\ba, \bC_\star, C_\mu)$. For $N \ge (\lambda/C_1)^2$ we then get upon
combining the triangle inequality with Theorems \ref{thm:main-expectations}--\ref{thm:conc}, that 
\[
\bP (  Z_N
 > 3 \lambda) \le 2 p_N(\lambda)
\,.
\]
Since $\sum_N p_N(\lambda) < \infty$ for any fixed $\lambda>0$, by Borel-Cantelli
$Z_N \stackrel{a.s.}{\to} 0$ as $N \to \infty$. Similarly, upon using the triangle 
inequality for $\|\cdot\|_q$ we get
from Theorems \ref{thm:main-expectations}-\ref{thm:conc} that 
\[
\big(\bE[|Z_N|^q]\big)^{1/q} 
\le C_1 N^{-1/2} +  2 \Big( C^q + \int_C^\infty q \lambda^{q-1} p_N(\lambda) d \lambda\Big)^{1/q} \,.
\]
Further, $N \mapsto p_N(\cdot)$ decrease pointwise on $[C,\infty)$, while for any  $q\ge 1$, the preceding 
integral is finite for all $N$ large enough. 
With $\{Z_N^q\}_N$ uniformly integrable, it follows that $Z_N \to 0$ also in $L^q$. 
\end{proof}

\subsection{Proof strategy}\label{sec:proof-strategy}
As mentioned in the introduction, traditional approaches to proving universality run into substantial difficulty when we apply them to diffusions with random coefficients. The dependence on specific entries of the random matrix are quite bad, as the dependence applies in the drift both through the $J_{ij}$, and through its effect on $\bX_t$, whose history evidently also depends on $J_{ij}$: this effect can exponentially amplify small differences; in fact, the exponential amplification is inherent to the problem at hand.  

At a high level, our strategy for proving Theorem~\ref{thm:main-expectations}, and the main novelty of the paper, is to leverage the independence of $\mu$
from $\bP_{\bJ},\bP_{\tbJ}$ by \emph{pulling back} $f(\bX_t)$ and $f(\tbX_t)$  
to properties of (time) derivatives of $f(\bX_t)$ evaluated at $t=0$. At the level of expectations, these derivatives can be seen as iterates of the infinitesimal generator applied to the function $F$, which can then be controlled by combinatorial
moment methods. The dominant contribution to the drift of $F$  comes from drift terms that are polynomials of degree at most two in $(J_{ij})_{ij}$. Since the first two moments of $\bP_{\bA}$ and $\bP_{\tbA}$ match, these terms do not contribute to the difference in expectations above. We emphasize that the approach does not need rely on an explicit solution to the \abbr{SDE} of~\eqref{eq:sde-def}, nor does it use  exponential control, or large deviations theory as in~\cite{DLZ19}, or refined estimates on the spectrum of $\bA$ as in the setting of~\cite{BADG01} where, crucially, the process has a rotational symmetry.   

Recall that the \abbr{SDE} defined in Eq.~\eqref{eq:sde-def} has infinitesimal generator $L$ that we split as follows (see e.g.,~\cite[Theorem 7.3.3]{Oksendal}): 
\begin{align}\label{eq:L-splitting}
L: = & \underbrace{\sum_{1\leq i,j\leq N}J_{ij}x_{i}\partial_{j}}_{L_{\bJ}}+\underbrace{\sum_{1\leq i,j\leq N}\Lambda_{ij}x_{i}\partial_{j}}_{L_{\bLambda}}+\underbrace{\sum_{1\leq j\leq N}h_{j}\partial_{j}}_{L_{\bh}}+\underbrace{
\sum_{1\leq j\leq N} \big(\sum_{0\le i \le N} \sigma_{ij}  x_i \big)^2
\partial_{j}\partial_{j}}_{L_{\Delta}}\,.
\end{align}
By Ito's formula, we have for every $f$, say, in $C^\infty(\bR^N)$,
\begin{align*}
\big|\bE[f(\bX_{t})]-\tbE[f(\tbX_{t})]\big|= & \big|\bE [P_t f(\bX_0)]-\tbE[P_t f(\bX_0)]\big|\,,
\end{align*}
where $P_t = P_t (\bJ)$ denotes the \emph{semi-group operator} 
\begin{align}\label{eq:semigroup}
    P_t f(\bx) := \bE_{\bB}[f(\bX_t)\mid \bX_0 = \bx]\qquad \mbox{with formal expansion}\qquad P_t = e^{tL}
\end{align}
in terms of the generator $L$. 
In order
to reduce the problem to a combinatorial question, we wish to Taylor expand the semi-group operator $P_t f = e^{tL}f$. As long as $f$ is smooth and the Taylor expansion converges absolutely---shown in Section~\ref{subsec:switch-expectation-sum}---this formal expansion is valid and we can  switch expectations over $\mu, \bP_\bJ,\bP_\tbJ$ with the sum, and compute expectations of powers of the generator $L$ acting on $f$. Namely, the difference in expectations is bounded by controlling (1) the size in $N$, and (2) the growth in $k$ of 
\begin{align}\label{eq:intro-diff-in-expectations}
    |\bE[(L^k f)(\bX_0)] - \tbE[(L^k f)(\bX_0)]|\,.
\end{align}
Expanding these terms as words in $L_\bJ,L_\bLambda,L_\bh,L_\Delta$, we 
observe that a non-zero difference between the two 
expectations in~\eqref{eq:intro-diff-in-expectations}, can only 
come from the summands (monomials in $\bJ,\bX,\bLambda, \bh,\bsigma$) satisfying 
\begin{itemize}
    \item Every $J_{ij}$ that is present, must appear at least twice. 
    \item At least one $J_{ij}$ must appear at least three times. 
\end{itemize}
This is because the means of $\bP_\bA,\bP_\tbA$ are zero, and the variances of $\bP_\bA$ and $\bP_\tbA$ match. A careful analysis of this combinatorial problem for the monomials eventually yields that the contributions from these monomials are, together, $O(N^{-1/2})$ in $N$, and $o(k!)$ in $k$: this computation is carried out in Section~\ref{subsec:difference-in-expectations}.

\begin{rem}
One may notice that in the case where $\bSigma(\bX_t)$ is constant so that $\bM_t$ is just a Brownian motion, we are left with a linear \abbr{SDS} and one could use this linearity in a more central way, to explicitly solve expectations of monomials in $(X_i(t))_i$ as Gaussian integrals and time integrals over words in $e^{ s \bJ}$ and $(X_i(t))_i$. 
If the system $\bX_t$ is invariant under rotations, then we can work in the coordinates of $\bJ$ so that it is diagonal and apply universality results for the spectrum of $\bJ$. Absent rotational symmetry, however, the natural step would be to Taylor expand $e^{s\bJ}$, at which point the expansion and the resulting combinatorics will be similar, and perhaps less transparent, than our generator based approach. Of course, for non-constant $\bSigma(\bX_t)$ as in Theorem~\ref{thm:main-expectations}, the \abbr{SDS} is non-linear, and such an approach would not generalize. 
\end{rem}

In Section~\ref{sec:more-general-observables}, we extend this bound on the difference in expectations of statistics $f$ to multi-time observables, then to statistics that contain the driving martingale terms 
and finally establish the universality at the level of expectation for observables of the form of~\eqref{eq:F-general-form}, as
stated in  Theorem~\ref{thm:main-expectations}. In Section~\ref{sec:concentration}, we 
adapt the approach of \cite{BADG06} to establish Theorem \ref{thm:conc}, namely, to 
show that the restricted class of observables of~\eqref{eq:possible-observables} concentrate around their expectations, by localizing to a set of large probability where $F$ is $O(N^{-1/2})$-Lipschitz in the triplet $(\bX_0,\bJ, (\bM_t)_{t\in [0,T]})$ and using Hypothesis~\ref{hyp:2}.

\subsection{Applications}\label{sec:applications}
In this section, we discuss systems for which Theorem~\ref{thm:main-expectations}--Corollary~\ref{cor:main} imply concrete universality results. All the examples that follow will be in the context of $\bSigma$ that is constant, i.e., $\sigma_{ij}= 0$ if $i\ne 0$, where both Theorems~\ref{thm:main-expectations}--\ref{thm:conc} apply. Among the examples with non-constant $\bSigma$, one which may be of interest is
a system of geometric Brownian motions interacting linearly through $\bJ$. 

We next describe two well-studied families of Markov processes/dynamical systems to which our results apply: Langevin dynamics and gradient flows on various energy landscapes (Hamiltonians) or loss functions. 

\subsubsection*{Langevin dynamics}
In the case where $\bJ$ and $\bLambda$ are symmetric matrices, and $\sigma_{0j}$
are identically one,~\eqref{eq:sde-def} corresponds exactly to the \emph{Langevin dynamics} for the \emph{Hamiltonian
}
\begin{align}\label{eq:Hamiltonian}
H(\bx)= & - \beta \sum_{1\leq i,j\leq N}(J_{ij}+\Lambda_{ij})x_{i}x_{j}- \beta \sum_{1\leq i\leq N}h_{i}x_{i}\,.
\end{align}
The linearity of the diffusion here corresponds to having a quadratic Hamiltonian. The Langevin dynamics is a reversible Markov process designed such that, when non-degenerate, its invariant measure on $\bR^N$ is given by $d\pi (\bx) \propto e^{ - H(\bx)} d\bx$. 
For Hamiltonians coming from spin glass theory, the Langevin dynamics has been analyzed at length in the case of Gaussian disorder, and found to have a varied and rich behavior; in \S\ref{subsec:ssk}, we explore this further in the context of a simple spin glass model, called the spherical \abbr{SK} model. 

\subsubsection*{Gradient flows}
The case where $\sigma_{0j}$ are identically zero---i.e., besides the randomness of $\bJ$ and, possibly, the initial data, the dynamics is deterministic---
also fits into
the framework of the paper. Here, given $\bJ$ and $\bX_0$, the law of the dynamics is taken to be
the delta function on the trajectory of the solution to the resulting  system of \abbr{ODE}'s. This corresponds to the \emph{gradient flow} on $H(\bx)$: in optimization and learning settings, e.g., the examples of Sections~\ref{subsec:example-Hopfield}--\ref{subsec:matrix-pca}, gradient descent and its many variants, are favored methods.

\medskip
We now turn to a few well-studied concrete problems to which our results are applicable.

\subsubsection{The (soft) spherical \abbr{SK} model}\label{subsec:ssk}
The dynamics of spin glasses are a canonical setting in which Markov processes with random coefficients are studied in their thermodynamic ($N\to\infty$) limit. The short-time ($N\to\infty$, then $T\to\infty$) behavior of Langevin dynamics, especially, in the context of spin glasses have been extensively studied in both the physics and math literature~\cite{CriSom92,Crisanti1993,CugKur93,BADG01,BAG1,BAG2,BADG06,DemboSubag,BGJ18a,BGJ18b}. Perhaps the most well-known mean field spin glass is the Sherrington--Kirkpatrick (\abbr{SK}) spin glass, where $N$ spins taking values in $\{+1,-1\}$ interact pairwise with one another, and their interaction strengths are moderated by ``coupling" parameters $J_{ij} = J_{ji}$ which are drawn i.i.d., say, Gaussian. 
We discuss a simplification of this known as the spherical \abbr{SK} model, which has been found to nevertheless exhibit some of the same phenomena. 

Take an i.i.d.\ symmetric matrix $\bJ=(J_{ij})_{ij}$ with law 
$\bP_{\bJ}$. The spherical \abbr{SK} model has \emph{Hamiltonian} 
\begin{align}\label{eq:ssk-hamiltonian}
H(\bx)= & \sum_{1\leq i,j\leq N}J_{ij}x_{i}x_{j}\qquad\mbox{for}\qquad \bx\in\bS^{N-1}(\sqrt{N})\,.
\end{align}
To avoid differential geometry on the sphere, it is sometimes preferable
to extend the Hamiltonian to all $\bx\in\bR^{N}$ (note that
the Hamiltonian is homogeneous so that dividing $\bx$ by the Euclidean norm $\|\bx\|/\sqrt N$
gives the same process on $\bS^{N-1}(\sqrt N)$). Instead of adding a non-linear confining force as is done in, e.g.,~\cite{BADG01}, we either add a linear confining force $F_K(x) = K x$, or have no confinement ($K = 0$) (the linearity of the system ensures no finite time blowup). Consider now the Langevin dynamics at \emph{inverse temperature} $\beta>0$ for the Hamiltonian of~\eqref{eq:ssk-hamiltonian}, corresponding to
$\bX_t = \bX_{t}^{(\beta)}$ solving the \abbr{SDS} 
\begin{align}\label{eq:Langevin-sde}
\begin{cases}
d\bX_{t} &=  -\nabla H(\bX_{t}) dt- F_{K}'(\|\bX_{t}\|^{2}/N) \bX_t dt+ \beta^{ -1/2} d\bB_{t}\\
\bX_{0} & \sim \mu
\end{cases}\,.
\end{align}
We also consider the \emph{gradient flow} where we take $\beta = \infty$, so that the Brownian motion term drops out: $\bX_t$ is then the (deterministic) dynamical system following the (random) gradient vector field of $H(\bx) + F_K(\|\bx\|^2/N)$. The following universality for the above system is an immediate corollary of Theorem~\ref{cor:main}. 
 
\begin{cor}\label{cor:ssk-universality}
Fix $\beta \in (0,\infty]$ and consider the \abbr{SDS}'s $\bX_t$ and $\tbX_t$ given by~\eqref{eq:Langevin-sde} for $\bA$
and $\tbA$ having mean zero, matching variance profiles $m_{ij} = \one\{i\neq j\}$. Suppose $\mu$ is independent of $\bP_{\bA}, \bP_{\tbA}$ and these satisfy Hypotheses~\ref{hyp:1}--\ref{hyp:2}. Then for $F$ as in~\eqref{eq:possible-observables} with $\cY, \cY'\in {\mathfrak F}$ and $\|\ba\|_\infty \le C_\ba$, for every $T<\infty$, 
\begin{align*}
\sup_{s,t\in [0,T]}\big|F (\bX_{s},\bX_{t})-F(\tbX_{s}, \tbX_{t})\big|\to0 & \qquad
\mbox{almost surely, and in $L^q$ for $q\ge 1$}\,.
\end{align*}
\end{cor}
As shown in~\cite{CugliandoloDean} and rigorously proved in~\cite{BADG01},  when $\bJ$ is Gaussian, the spherical \abbr{SK} model, or the soft spherical \abbr{SK} Model with confining potential $F$ satisfying $F(x)/x\to \infty$ as $x\to\infty$, exhibits a sharp \emph{aging} transition. Informally, aging is defined as the notion that the older a system gets, the more it remembers its past; formally, it corresponds to a transition in the behavior of the auto-correlation,
\begin{align*}
    C_N(s,t) := \frac 1N \sum_{i\le N} X_i(s)X_i(t)\,,
\end{align*}
between a (\abbr{FDT}) regime where $C_N(s,t) \sim \Phi (t-s)$ and an aging regime where $C_N(s,t) \sim \Phi(\frac ts)$ for large $s,t$. In~\cite{BADG01}, it was established that for $\bJ$ having rotationally invariant law, e.g., a 
\abbr{GOE} matrix, $C_N(s,t)$ solves a non-linear equation~\cite[Eq.~(2.16)]{BADG01}, which exhibits exactly this type of transition at some $\beta_{\rm ag}$. Our results allow us to read off universality for this limiting behavior, as formalized in the following corollary.   

\begin{cor}\label{cor:ssk-aging}
	Consider the Langevin dynamics for the soft spherical \abbr{SK} model, as defined in~\eqref{eq:Langevin-sde} where $\bP_\bA$ is a Wigner matrix satisfying Hypothesis~\ref{hyp:2}, the confinement is $F_K(x) = Kx$ for some $K> \bE[ \|\bJ\|_{2\to 2}]$, and the initialization $\mu$ is e.g., standard Gaussian, independent of $\bP_\bA$. Then, for every $\beta\in (0,\infty]$ and every $T<\infty$, the limit $(\lim_{N\to\infty} C_N(s,t))_{s,t\in [0,T]}$ exists, and satisfies~\cite[Eq.~(2.16)]{BADG01}.
	
	In the specific case of $\beta= \infty$, the conclusions of~\cite[\S3.2.2]{BADG01} apply, and the solution exhibits aging: i.e., there is a $\gamma>0$ (specified therein) such that for every $\lambda>1$, $$\lim_{s\to\infty} \lim_{N\to\infty}\frac{C_N(s,\lambda s)}{\sqrt{C_N(s,s) C_N(\lambda s,\lambda s)}} \approx (\lambda - 1)^{-\gamma}\,.$$ 
\end{cor} 

\begin{proof} 
For the first statement, while \cite[Theorem~2.6]{BADG01} is stated for confinement $F$ growing super-linearly, following the proof one sees that it is only used to localize the process, for which it suffices for $K$ to exceed $\|\bJ\|_{2 \to 2}$ (which for Wigner matrices is a.s. less than $2+\epsilon$ for any $\epsilon>0$). The first part of the corollary therefore follows from Corollary~\ref{cor:ssk-universality} together with the result of~\cite[Theorem 2.6]{BADG01} showing that for $\bA$ standard normal, $C_N(s,t)$ satisfies~\cite[(2.16)]{BADG01}. 

For concreteness, the analysis of the limiting equation~\cite[(2.16)]{BADG01} and 
the derivation of the aging transition is carried out in~\cite{BADG01} only
for a specific choice of quadratic $F$. One could in principle perform the same analyses with other choices of $F$ including $F=F_K$ that is linear, corresponding to the case we consider, and understand the limiting behavior of $C_N(s,t)$ as $N\to\infty$ then $s,t\to\infty$ as $\beta$ varies. 
We do not pursue this, and instead notice that in the specific case of $\beta = \infty$, the homogeneity allows us to disregard the choice of the confining potential and obtain universality for the zero-temperature aging behavior.  To see this, since $H(x)$ is a homogeneous polynomial, if $\beta = \infty$, we see that $d\bX_t$ is a constant multiple (for a constant depending only on $\|\bX_t\|$) of  $d(\bX_t/ \|\bX_t\|)$. Therefore, at $\beta = \infty$, the projection of the dynamics~\eqref{eq:Langevin-sde} onto the sphere $\bS^{N-1}(\sqrt N)$ matches the projection of the Langevin \abbr{SDS} of~\cite{BADG01}, regardless of the choice of confining potential used therein. We apply Corollary~\ref{cor:ssk-universality} first to deduce that $\lim_{s\to\infty} \lim_{N\to\infty} C_N(s,s)=: C_\infty$ is the same for Gaussian and non-Gaussian $\bP_{\bA}$. Then applying it to $C_N(s,\lambda s)$, we find that the $N\to\infty$ limit of the normalized auto-correlation is the same for Gaussian and non-Gaussian $\bP_{\bA}$, and it is further independent of the choice of confining potential: as such for any $\bP_{\bA}$, it has the same $N\to\infty$ limit as in~\cite{BADG01}. 
\end{proof}

\begin{rem}\label{rem:p-spin}
It would be of interest to consider similar Langevin dynamics for the spherical or soft spherical $p$-spin glass models for $p>2$. Permitting higher order interactions gives rise to a wealth of more complicated models and different behavior. At the level of the off-equilibrium Langevin dynamics, these lead to the famous Cugliandolo--Kurchan/Crisanti--Horner--Sommers limit of coupled integro-differential equations for $C_N(s,t)$ and  an integrated response $\chi_N(s,t)  = \frac 1N \sum_i X_i(s) B_i(t)$~\cite{CriSom92,Crisanti1993,CugKur93,BADG06,DGM07,Gui07,DemboSubag}, as well as the evolution of other observables e.g., the Hamiltonian and its square gradient~\cite{BGJ18a}. Our combinatorial framework suggests that the differences in expectations (over $p$-tensors $\bJ$ and $\tbJ$) of averaged observables are microscopic, as long as there is a non-linear confining potential to prevent finite-time blowup. The complication is in the fact that the two non-linearities (from the interactions, and the confining potential) cancel out, but these cancellations are not easily seen in the Taylor series obtained by expanding in powers of the generator; thus we are not able to show that this series is absolutely summable and exchange the infinite sum with its expectation. 
\end{rem}

\subsubsection{Symmetric and asymmetric Hopfield networks}\label{subsec:example-Hopfield}
Let us also mention a different context in which diffusions of the form of~\eqref{eq:sds-intro} appear. Hopfield networks were introduced by~\cite{Hopfield} and have become one of the simplest and most fundamental examples of neural networks. In this model, a set of $N$ neurons $(X_i)_i$ are either active $\{+1\}$ or inactive $\{-1\}$ depending on whether the neuron $X_j$'s input $\sum J_{ij} X_i$, for some weights $\bJ = (J_{ij})_{i,j}$, exceeds a deterministic threshold $h_i$. This model was introduced in the symmetric setting, but has since been analyzed extensively both in symmetric and asymmetric setups~\cite{Hertz-et-al,CrisantiSompolinsky,Xuetal}. 

One typically initializes the neurons at some pre-determined state independent of $\bJ$, e.g., all inactive/active, or uniformly at random, and tracks their time-evolution, whereby each neuron activates/de-activates at some rate, depending on the relationship between its input and threshold. Though there are many ways this is implemented, one is to soften the problem to continuous state space, either to the sphere, or to full-space and add in stochasticity by running some Langevin dynamics. This is the approach pursued in~\cite{CrisantiSompolinsky} as well as e.g.,~\cite{Xuetal}. Then, with a linear confining force, our results imply universality for both for the symmetric and asymmetric Langevin dynamics (and gradient flow) of general Hopfield networks: this includes universality for observables capturing the energy/loss in the network, its square gradient, and its ``memory".

\subsubsection{Rayleigh quotient minimization for random matrices}\label{subsec:matrix-pca}
We conclude with a  related optimization problem in high dimensions: that of optimizing the Rayleigh quotient of a random matrix $\bJ$ with a certain mean and variance profile. Maximizing the Rayleigh quotient is an efficient way to find the top eigenvector and eigenvalue of the random matrix via local iteration, e.g., either gradient descent or Langevin dynamics at low temperatures (large $\beta$). To place this in the framework of~\eqref{eq:sde-def}, take $H(\bx) = \langle \bx, \bJ \bx\rangle$ and either no confining force or $F_K' = K$ for some $K > \|\bJ\|_{2 \to 2}$ in~\eqref{eq:Langevin-sde}. In the situation where the matrix ensemble is rotationally invariant, e.g., the \abbr{GOE},  the limiting trajectories of, say, $H(\bX_t)$ for the gradient flow/Langevin dynamics can be explicitly solved (by diagonalization). Corollary~\ref{cor:main} implies these limiting trajectories will be universal, and thus, match the limiting trajectories obtained when $\bJ$ is not Gaussian.  In~\cite{BLM15,ChenLamUniversality}, similar universality results were described for an \abbr{AMP} approach to finding the top eigenvalue/eigenvector of $\bJ$. 

\medskip
\noindent \textbf{Acknowledgments.}  
The authors thank the anonymous referee for useful comments, and Ramon van Handel and Ofer Zeitouni for helpful conversations. This project was supported in part by NSF grants \#DMS-1613091, \#DMS-1954337 (A.D.), and by the Miller institute for basic research in science (R.G.).

\section{Universality of expectations of monomial observables}\label{sec:universality-expectations}
In this section, we prove that two solutions $\bX$ and $\tbX$ of~\eqref{eq:sde-def} driven by $\bJ$ and $\tbJ$ are such that expectations of observables of the form~\eqref{eq:possible-observables} are universal, as long as $\bA$ and $\tbA$ have the same variance profiles. As discussed in Section~\ref{sec:proof-strategy}, we 
reduce differences in expectations to combinatorial calculations by expanding
the Markov transition semi-group of the process $\bX_t$ in terms of its generator, an approach
for proving universality in randomly driven dynamical systems which is the key contribution of this paper. 

For the entirety of this paper, we will take two distributions $\bP_{\bA}$ and $\bP_{\tbA}$ on $\bA$ and $\tbA$ that are mean zero and have the same, uniformly bounded, variance profiles $\bm = \tbm$. Recall that $\bP_{\bA}$ and $\bP_{\tbA}$ are either fully independent or symmetric ensembles. For conciseness, we present our results in the case of fully independent (in particular, not symmetric). The case where they are symmetric is handled mutatis mutandis and only induces a few constant factors in certain estimates
(see Remark~\ref{rem:symmetric-independent} for more on these minimal modifications).

\subsection{Main result on difference in expectations}
The observables in Theorem~\ref{thm:main-expectations} are composed of polynomials in $\bJ$ and $\bX$, as well as $\bM$. We first establish the universality of expectations for general monomials in $\bJ$ and $\bX$ via a combinatorial moment matching type of argument. In
Section~\ref{sec:more-general-observables} such universality is reduced for monomials that additionally 
involve the martingale, to that of monomials only  in $\bJ$ and $\bX$. 
 
More precisely, the statistics we consider throughout this section are of the following form. Fix any $s$ (not necessarily distinct) pairs $\balpha =( \alpha_1,\ldots,\alpha_s)$ where each $\alpha_k = i_kj_k$, and $r$-tuple (not necessarily distinct) $\boldsymbol \gamma = (\gamma_1,\ldots,\gamma_r)$ where each $\gamma_i \in \{1,\ldots,N\}$. Then consider observables $f_{\balpha,\bgamma}(\bx)$ of the form 
\begin{align}\label{eq:f-balpha-gamma-def}
f_{\balpha,\bgamma}(\bx)= \prod_{k=1}^{s} J_{\alpha_k} \prod_{l=1}^r x_{\gamma_l}\,.
\end{align}
For an $s$-tuple of pairs $\balpha$, let
\begin{itemize}
    \item $I_{\balpha}$ count the number of distinct pairs in $\balpha$, i.e., $I_\balpha = |\{\alpha_1,\ldots, \alpha_s\}|$,
    \item $I_{\balpha,1}$ count the number of $(\alpha_k)_k$ which appear exactly once in $\balpha$, and
    \item $I^+_{\balpha,1}$ equal $I_{\balpha,1}$ plus the indicator 
    that no pair appears more than twice in $\balpha$.  
\end{itemize} 
Our bound on the distance between the expectations of $f_{\balpha,\bgamma}(\bX_t)$ and $f_{\balpha,\bgamma}(\tbX_t)$ depends on $\balpha$, $\bgamma$ and the laws $\mu$, $\bP_{\bA}$, $\bP_{\tbA}$ 
only through ${\bC}_\star$, $\Cmu$, $s$, $r$ and $I^+_{\balpha,1}$. More precisely, we derive here the following.

\begin{prop}\label{prop:difference-in-expectations}
There exists $C= C(r,s,T, \bC_\star,\Cmu (r))$ such that for every $T,r,s \ge 0$, 
every $s$-tuple of pairs $\balpha$ and every $r$-tuple $\bgamma$, if
$\bP_{\bA}$, $\bP_{\tbA}$ and $\mu$ satisfy Hypothesis~\ref{hyp:1}, then  
\begin{align*}
\sup_{t \in [0,T]} \big|\bE [f_{\balpha, \bgamma} (\bX_t) ] 
-\tbE [f_{\balpha, \bgamma}(\tbX_t)]\big| \leq C N^{-(s+I^+_{\balpha,1})/2} \,.
\end{align*} 
Observe that in the case $s=0$, the right-hand side is $C N^{-1/2}$. 
\end{prop}

\begin{rem}
The above theorem shows that having more distinct $J$'s in the observable, decreases the difference in expectations by more than $N^{-s/2}$ as would be expected from the typical size of $J_{ij}$. This should be expected due to \abbr{CLT}-type cancellations:  one way to motivate this scaling is by recalling averaged statistics which have $J$ in them, in the context of the spherical \abbr{SK} model, e.g., the most relevant being 
\[
\frac{H(\bx)}{N}=  \frac{1}{N^{3/2}}\sum_{1\leq i,j\leq N} A_{ij}x_{i}x_{j}\qquad \mbox{and}\qquad \frac{|\nabla H(\bx)|^2}{N} = \frac 1N\sum_{1\leq i\leq N} G_i^2(\bx)= \frac 1N \sum_{1\leq i \leq N} \Big(\sum_{1\le j\le N} \frac{1}{\sqrt N} A_{ij}x_j\Big)^2\,.
\]
(Notice that these statistics are \emph{not}
rescaled by the number of order-one sized monomials; but they remain on the $O(1)$ scale due to additional cancellations from $(J_{ij})$). This gain in the scaling has to be visible at the level of the difference in expectations under $\bP$ and $\tbP$ in order to hope for universality for such statistics.
\end{rem}

Recall from Section~\ref{sec:proof-strategy} that our high level strategy is to reduce the expectations of statistics of the solution $\bX_t$ of the \abbr{sds} 
 to combinatorial calculations in terms of mixed moments of $\bJ$ and $\bX_0$. This is possible by writing $\bE_{\bB}[f(\bX_t)]$ as $P_t f (\bX_0)$ and then Taylor expanding $P_t = e^{tL}$ where $L$ is the generator for the process $\bX_t$ as defined in~\eqref{eq:L-splitting}. In order for this expansion to be valid, and therefore our approach to be permissible, we need the Taylor expansion for $e^{tL}$ to converge absolutely, for each fixed $N$. 
In the next sub-section, we show that indeed with $\mu, \bP_\bA, \bP_\tbA$ satisfying Hypothesis~\ref{hyp:1}, for each fixed $N$, the infinite series corresponding to $P_t f$ converges absolutely, so we can follow this plan. 

Before proceeding further, we make the following notational remark. 

\medskip
\noindent \textbf{Notational comment on set and sequence differences.} For sets $\{b_1,\ldots,b_m\} \subset \{a_1,\ldots,a_n\}$, we let $\{a_1,\ldots,a_n\}\setminus \{b_1,\ldots,b_m\}$ denote the set difference as usual. Frequently we deal with tuples, or sequences in which the order does not matter. For two such tuples $(a_1,\ldots,a_n)$ and $(b_1,\ldots,b_m)$ (where of course there may be repetitions in each sequence), we denote by $(a_1,\ldots,a_n)\setminus (b_1,\ldots,b_m)$ the difference wherein for each $b_i$ appearing in $\{a_1,\ldots,a_n\}$ we only remove one of its appearances---say the first one---from $(a_1,\ldots,a_n)$. We also define $(a_1,\ldots,a_n) \amalg (b_1,\ldots,b_m)$ to be the concatenation given by  $(a_1,\ldots,a_n,b_1,\ldots,b_m)$.

\subsection{Switching the expectation and the infinite series}\label{subsec:switch-expectation-sum}
The goal of this sub-section is to prove the following absolute convergence result.

\begin{prop}\label{prop:absolute-convergence}
Suppose $\bP_\bA$ and $\mu$ satisfy Hypothesis \ref{hyp:1}. Then, 
there exists finite $N_o=N_o(r,T,\bC_\star)$ such that for every $N \ge N_o$, every $T<\infty$,
every $s$-tuple of pairs $\balpha$, and every $r$-tuple of indices $\bgamma$, we have  
\begin{align*}
\sum_{k\geq0}\frac{T^{k}}{k!}\bE\big[\big|L^{k}f_{\balpha,\bgamma}(\bX_{0})\big|\big]<\infty\,.
\end{align*}
\end{prop}

As a consequence of Proposition~\ref{prop:absolute-convergence} 
and Fubini--Tonelli, we may use the following expansion. 
\begin{cor}\label{cor:switch-expectation-sum}
Suppose $\bP_{\bA}$, $\bP_{\tbA}$, $\mu$ satisfy Hypothesis~\ref{hyp:1}.
Setting $L$ and $\tilde L$ for their generators, we have that 
\begin{align*}
\bE[f_{\balpha,\bgamma}(\bX_{t})]-\tbE[f_{\balpha,\bgamma}(\tbX_{t})]= & \sum_{k\geq0}\frac{t^{k}}{k!}\Big(\bE[L^{k}f_{\balpha,\bgamma}(\bX_{0})]-\tbE[\tilde L^{k}f_{\balpha,\bgamma}(\bX_{0})]\Big)\,,
\end{align*}
for every $N \ge N_o(r,T,{\bC}_\star)$, every $t <\infty$, and every $s$-tuple of pairs $\balpha$ and
$r$-tuple of indices $\bgamma$.
\end{cor}

Proceeding hereafter  to prove Proposition~\ref{prop:absolute-convergence},  we 
fix $r,s,\balpha$ and $\bgamma$, and set $f=f_{\balpha,\bgamma}$. 
Aiming for upper bounds on $\bE[|L^{k}f(\bX_0)|]$ which are summable against $T^{k}/k!$, we first
utilize 
\eqref{eq:L-splitting}
to expand $L^{k}$ as a sum over the $4^{k}$ words $W$ in the letters $\{L_{\bJ},L_{\bLambda},L_{\bh},L_{\Delta}\}$
and thereby get the bound
\begin{align}\label{eq:L-decompose-words}
\bE\big[|L^{k}f(\bX_{0})|\big]\leq & 4^k \sup_{W\in\{L_{\bJ},L_{\bLambda},L_{\bh},L_{\Delta}\}^{k}}\bE\big[|Wf(\bX_{0})|\big] \,,
\end{align}
where for every $\bx \in \bR^N$, $Wf(\bx)$ should be understood as $(W_k \cdots W_2 W_1 f)(\bx)$. For every
word $W\in\{L_{\bJ},L_{\bLambda},L_{\bh},L_{\Delta}\}^{k}$, let $k_{\bJ}=k_\bJ(W)$
denote the number of $L_{\bJ}$'s that appear in $W$, and similarly
define $k_{\bLambda}$, $k_{\bh}$, and $k_{\Delta}$, so that $k_{\bJ}+k_{\bLambda}+k_{\bh}+k_{\Delta}=k$ 
and the following structural decomposition of $Wf$ holds. 

\begin{claim}\label{claim:monomial-structure}
For any word $W\in\{L_{\bJ},L_{\bLambda},L_{\bh},L_{\Delta}\}^{k}$ with
$k_{\bJ},k_{\bLambda},k_{\bh},k_{\Delta}$ occurrences of the corresponding symbols,
$Wf$ can be expressed as a sum of (not necessarily distinct) monomials of the form 
\begin{align}\label{eq:monomial-structure}
\phi_{\bbeta,\bbeta',\bzeta',\bzeta,\bxi}(\bx)=\prod_{i=1}^{s}J_{\alpha_i}\prod_{\ell=1}^{k_{\bJ}}J_{\beta_{\ell}}\prod_{\ell=1}^{k_{\bLambda}}\Lambda_{\beta'_{\ell}}\prod_{\ell=1}^{k_{\bh}}h_{\zeta'_{\ell}}\prod_{\ell=1}^{2 k_\Delta} \sigma_{\zeta_\ell} \prod_{\substack{\ell=1}}^{r}x_{\xi_{\ell}}\,,
\end{align}
$\bbeta,\bbeta',\bzeta$ denote the collection of pairs $(\beta_\ell)_{\ell\leq k_
\bJ}$, 
$(\beta'_{\ell})_{\ell\leq k_\bLambda}$, $(\zeta_\ell)_{\ell\leq 2 k_\Delta}$, while
$\bzeta',\bxi$ denote the sequences $(\zeta'_\ell)_{\ell\leq k_\bh}$, $(\xi_\ell)_{\ell\leq r}$
and hereupon we adopt the convention $x_0 \equiv 1$, allowing for $\xi_\ell = 0$ 
as well as $\zeta_\ell \in (0j)_j$.
\end{claim}

In view of Hypothesis~\ref{hyp:1} on $\bP_\bA$ we have that for every $N$, $\ell \ge 0$, and index pair $\alpha$, 
\[
\bE[|J_{\alpha}|^{\ell+1}] \le \ell! \, \Big(\frac{\CA}{N}\Big)^{(\ell+1)/2} \,.
\]
Thus, if $I_{\balpha \amalg \bbeta}$ distinct index pairs appear at multiplicities $(n_{\ell}+1)_{\ell \le I_{\balpha \amalg \bbeta}}$ in the sequence $\balpha \amalg \bbeta$ 
of length $k_\bJ+s$, then by the independence of $(J_\alpha)_\alpha$,
\[
\bE\Big[\Big|\prod_{i=1}^{s}J_{\alpha_i}\prod_{\ell=1}^{k_{\bJ}}J_{\beta_{\ell}}\Big|\Big] \le 
\Big(\frac{\CA}{N}\Big)^{(k_\bJ+s)/2} \,
\prod_{\ell=1}^{I_{\balpha \amalg \bbeta}} n_\ell !  \, \,.
\]
Consequently, with $\bX_0$ independent of $\bJ$ we have in view of the assumed 
bounds on $(\Lambda_{ij})_{i,j}$ $(\sigma_{ij})_{i,j}$ and $(h_{i})_{i}$, that for any term of the form 
\eqref{eq:monomial-structure} with $I_{\bzeta}$ entries such that $\zeta_\ell \not\in (0j)_j$,
\begin{align}\label{eq:monomial-crude-bound}
\bE \Big[\big|\phi_{\bbeta,\bbeta',\bzeta',\bzeta,\bxi}(\bX_0)\big|\Big]
\leq & 
\Big( \frac{\CA}{N} \Big)^{(k_\bJ+s)/2}
\Big(\frac{C_\bLambda}{\cN_{\bLambda}}\Big)^{k_{\bLambda}}
\Big(\frac{C_\bsigma^{2k_\Delta}}{\cN_\bsigma^{I_{\bzeta}}} \Big)
C_\bh^{k_{\bh}}  
\sup_i \{ \bE \big[|X_{i}(0)|^{r}] \} \, \prod_{\ell=1}^{I_{\balpha \amalg \bbeta}} n_\ell !  
\nonumber \\
\le & \Cmu (r) {\bC}_\star^{s}  \frac{{\bC}_\star^{k}}{N^{(k_\bJ+s)/2} \cN_\bLambda^{k_{\bLambda}}
\cN_\bsigma^{I_{\bzeta}}}  \prod_{\ell=1}^{I_{\balpha \amalg \bbeta}} n_\ell !  \,,
\end{align}
using in the last inequality also \eqref{def:Cmu} from Hypothesis~\ref{hyp:1} 
on $\mu$, and the definition of $\bC_\star$.

Our next result is a first step in  controlling the number of monomial terms that can appear 
in the expansion of each word $W\in \{L_{\bJ},L_{\bLambda},L_{\bh},L_{\Delta}\}^{k}$. 
\begin{lem}\label{lem:L-number-of-terms}
For every $k_\bJ, k_\bLambda, k_\bh, k_\Delta$ and every $\bbeta,\bbeta',\bzeta',\bzeta,\bxi$,  if we let $\phi=\phi_{\bbeta,\bbeta',\bzeta',\bzeta,\bxi}$ be as in~\eqref{eq:monomial-structure}, then $L_\bh \phi$, 
$L_\bJ \phi$,  $L_{\bLambda} \phi$ and $L_{\Delta} \phi$ can each be expressed as a sum of at most 
$r$, $r N$, $r \cN_{\bLambda}$ and $r \cN_{\bsigma}^2$ 
many 
such monomials, respectively, each of the same form 
(with possibly different $\bbeta,\bbeta',\bzeta',\bzeta,\bxi$) as~\eqref{eq:monomial-structure},
with the respective $k_{\bJ},k_{\bLambda},k_{\bh}$ or $k_{\Delta}$ increased by one. 
\end{lem}
\begin{proof}
Fixing $k_{\bJ},k_{\bLambda},k_{\bh},k_{\Delta}$ which sum up to $k$, we proceed by separately considering the effect each of $L_\bh\phi$, $L_{\bJ}\phi$, $L_{\bLambda}\phi$ 
and $L_{\Delta}\phi$ has on the monomial $\phi$.
First, 
\begin{align}\label{def:L-H-phi}
(L_{\bh}\phi) (\bx) = \prod_{\ell=1}^{s}J_{\alpha_\ell}\prod_{\ell=1}^{k_{\bJ}}J_{\beta_{\ell}}\prod_{\ell=1}^{k_{\bLambda}}\Lambda_{\beta'_{\ell}}\prod_{\ell=1}^{k_{\bh}}h_{\zeta'_{\ell}}\prod_{\ell=1}^{2k_\Delta} \sigma_{\zeta_\ell} \sum_{j=1}^N h_{j}\partial_{j}\Big(\prod_{\substack{\ell=1}
}^{r}x_{\xi_{\ell}}\Big)\,,
\end{align}
with non-zero contribution only from $j \in \bxi$, yielding at most $r$ non-zero terms. To each of
these corresponds a monomial of the form of \eqref{eq:monomial-structure}, for 
$k_{\bh}\mapsto k_{\bh+1}$,  $\bzeta' \mapsto \bzeta' \amalg (j)$ and $\bxi \mapsto (\bxi \setminus (j))\amalg(0)$.  Next,
\begin{align}\label{def:L-J-phi}
(L_{\bJ}\phi) (\bx) = & \prod_{\ell=1}^{s}J_{\alpha_\ell}\prod_{\ell=1}^{k_{\bJ}}J_{\beta_{\ell}}\prod_{\ell=1}^{k_{\bLambda}}\Lambda_{\beta'_{\ell}}\prod_{\ell=1}^{k_{\bh}}h_{\zeta'_{\ell}}\prod_{\ell=1}^{2 k_\Delta} \sigma_{\zeta_\ell} \sum_{i,j=1}^N J_{ij}x_{i}\partial_{j}\Big(\prod_{\substack{\ell=1}
}^{r}x_{\xi_{\ell}}\Big) \,,
\end{align}
with non-zero contribution only when $j\in \bxi$. With $i \le N$ the total number of resulting 
non-zero monomials is now at most $rN$, each having the stated form with
$k_{\bJ}\mapsto k_{\bJ}+1$, $\bbeta\mapsto \bbeta\amalg (ij)$ and  $\bxi 
\mapsto (\bxi \setminus (j)) \amalg (i)$.  Likewise, we have  that
\begin{align}\label{def:L-Lambda-phi}
(L_{\bLambda} \phi) (\bx)= & \prod_{\ell=1}^{s}J_{\alpha_\ell}\prod_{\ell=1}^{k_{\bJ}}J_{\beta_{\ell}}\prod_{\ell=1}^{k_{\bLambda}}\Lambda_{\beta'_{\ell}}\prod_{\ell=1}^{k_{\bh}}h_{\zeta'_{\ell}}\prod_{\ell=1}^{2 k_\Delta} \sigma_{\zeta_\ell} \sum_{i,j=1}^N \Lambda_{ij}x_{i}\partial_{j}\Big(\prod_{\substack{\ell=1}
}^{r}x_{\xi_{\ell}}\Big)\,,
\end{align}
with non-zero contributions only for $j\in\bxi$. Enumerating over $i \le N$, gives now 
at most $r\cN_\bLambda$ non-zero monomials, of the stated form, with $k_{\bLambda}\mapsto k_{\bLambda}+1$, $\bbeta'\mapsto \bbeta'\amalg (ij)$ 
and $\bxi \mapsto (\bxi\setminus(j)) \amalg (i)$. Finally,
\begin{align}\label{def:L-Delta-phi}
(L_{\Delta}\phi) (\bx) = & \,\prod_{\ell=1}^{s}J_{\alpha_\ell}\prod_{\ell=1}^{k_{\bJ}}J_{\beta_{\ell}}\prod_{\ell=1}^{k_{\bLambda}}\Lambda_{\beta'_{\ell}}\prod_{\ell=1}^{k_{\bh}}h_{\zeta'_{\ell}}\prod_{\ell=1}^{2 k_\Delta} \sigma_{\zeta_\ell}  \sum_{j=1}^N \big(\sum_{i,i' = 0}^N \sigma_{ij} \sigma_{i'j} x_i x_{i'}\big) \partial_{j}\partial_j \Big(\prod_{\substack{\ell=1}
}^{r}x_{\xi_{\ell}}\Big) \,,
\end{align}
is non-zero only for the summands in which $j\in\bxi$. Enumerating over $0 \le i,i' \le N$ 
(recalling the convention that $x_0\equiv1$), gives at most $r\cN_\bsigma^2$ non-zero monomials, 
of the stated form, with $k_{\Delta}\mapsto k_{\Delta}+1$, $\bzeta \mapsto \bzeta \amalg (ij) \amalg (i'j)$
and $\bxi \mapsto (\bxi \setminus (j,j)) \amalg (i,i')$.
\end{proof}

Fixing $N$, $k$, an $s$-tuple of pairs $\balpha$, an $r$-tuple of indices $\bgamma$ 
and $W\in \{L_\bJ, L_\bLambda, L_\bh, L_\Delta\}^{k}$, upon inductively applying 
Lemma~\ref{lem:L-number-of-terms}, we are able to express $Wf$ as the sum of at most 
\begin{equation}\label{eq:crude-bd}
r^{k} N^{k_\bJ} \cN_{\bLambda}^{k_\bLambda} \cN_{\bsigma}^{2k_\Delta} \,,
\end{equation}
many  non-zero monomials of the form of~\eqref{eq:monomial-structure}.  
Recall that for a monomial $\phi$, we use $I_\bzeta$ for the number of $\zeta_{\ell}\notin (0j)_{j}$, $I_\balpha$ for the number of distinct pairs in $\balpha$, $I_{\balpha \amalg \bbeta}$ for the number of distinct pairs in $\balpha \amalg \bbeta$, and introduce $I_\star = I_{\balpha \amalg \bbeta} - I_\balpha$, which counts the number of distinct pairs in $\{\bbeta\} \setminus \{\balpha\}$. A
careful examination of the proof of Lemma \ref{lem:L-number-of-terms}, yields 
the following significant refinement upon the crude bound of \eqref{eq:crude-bd}.
\begin{prop}\label{prop:monomial-counting-crude} 
Fix $N$, $r,s,k \ge 0$, an $s$-tuple of pairs $\balpha$, an $r$-tuple of indices $\bgamma$, and
$W\in \{L_\bJ, L_\bLambda, L_\bh, L_\Delta\}^{k}$.
Then, 
of the monomials in such expansion of $Wf$, at most 
\begin{equation}\label{eq:bd-monomial-number} 
\, {k_\bJ \choose I_\star, n_1,\ldots,n_{I_{\balpha \amalg\bbeta}}} \, {2 k_\Delta \choose I_\bzeta} \,
r^k \, N^{I_\star}\, 
\cN_{\bLambda}^{k_\bLambda} \cN_{\bsigma}^{I_{\bzeta}}
\end{equation}
have $I_{\bzeta}$ elements of $\bzeta$ with $\zeta_\ell \not\in (0j)_j$, and the
$I_{\balpha \amalg \bbeta}=I_\balpha+I_\star$ distinct pairs in $\balpha \amalg \bbeta$ appear in multiplicities 
$\{n_\ell + \one_{\{\ell > I_\balpha\}}\}_{\ell\le I_{\balpha\amalg\bbeta}}$ within the sequence $\bbeta$ of length $k_\bJ$. (N.b.\ we ordered the $(n_\ell)$ with multiplicities in $\bbeta$ of the distinct pairs of $\balpha$ appearing first, and the multiplicities in $\bbeta$ of the remaining $I_\star$ distinct pairs next.)
\end{prop}
\begin{proof} The first improvement in \eqref{eq:bd-monomial-number} over \eqref{eq:crude-bd} 
is from observing that the growth factor $\cN_{\bsigma}$ applies only in those 
$I_\bzeta$ of the $2 k_\Delta$ applications of $L_\Delta$ within $W$ which have 
led to an element $\zeta_\ell \not\in (0j)_j$ (see \eqref{def:L-Delta-phi}), and that there 
are at most ${2 k_\Delta \choose I_\bzeta}$ ways to choose which $I_\bzeta$ elements of
$\bzeta$ are not from the $0$-th row of $\bsigma$. 

Similarly, the growth factor 
$N$ in counting the number of monomials after applying $L_\bJ$ is only relevant 
during the $I_\star$ applications of $L_\bJ$ within $W$ in which a new pair
$(ij)$ is selected (see \eqref{def:L-J-phi}). The left-most term in 
\eqref{eq:bd-monomial-number} counts the number of ways to select the locations of these
$I_\star$ new elements within the $k_\bJ$ long sequence $\bbeta$, and thereafter to 
partition the remaining $k_\bJ-I_\star$ consistently with having the prescribed 
$n_\ell \ge 0$ repeats for each of the $I_{\balpha\amalg\bbeta}$ distinct pairs in question. Putting all this 
together yields the stated bound \eqref{eq:bd-monomial-number}
on the number of relevant monomials in the expansion of $W f$. 
\end{proof}

\begin{proof}[\textbf{\emph{Proof of Proposition~\ref{prop:absolute-convergence}}}]Combining Proposition~\ref{prop:monomial-counting-crude} with the 
bound \eqref{eq:monomial-crude-bound} we deduce that for any word $W$ of length $k$
and any $\balpha$ whose $I_\balpha$ distinct terms appear in multiplicities $(c_\ell)_{\ell\le I_\balpha}$,
\begin{equation}\label{eq:reza-added}
\bE[|W f(\bX_0)|] \le 
\Cmu (r) {\bC}_\star^{s} \,  k_\bJ! \,  \frac{(4 r {\bC}_\star)^{k}}{N^{(k_\bJ+s)/2}}
\sum_{I_\star =0}^{k_\bJ} \frac{N^{I_\star}}{I_\star!} \sum_{(n_\ell)_{\ell \le I_{\balpha\amalg\bbeta}}} \,
\prod_{\ell =1}^{I_{\balpha}} \frac{(n_\ell+c_\ell -1)!}{n_\ell !} \,,
\end{equation}
where the inner sum is over all partitions of $k_\bJ-I_\star$ into $I_{\balpha\amalg\bbeta}$ indistinguishable
integers $n_\ell \ge 0$. Since $ \sum_\ell c_\ell = s$ and $n_\ell +c_\ell \le  k_\bJ+s$ for all $\ell$, 
the right-most 
product is at most $(k_\bJ+s)^s$. Further, the number of $(n_\ell)_\ell$ considered here is at 
most the number of integer partitions of $k_\bJ$, which grows slower 
than $e^{k_\bJ}$ (c.f. the Hardy-Ramanujan asymptotic partition formula~\cite{HardyRamanujan}). Thus, we find that
for $C(r,s,\Cmu,{\bC}_\star)$ finite and any word $W$ of length $k$,
\begin{equation}\label{eq:w-bd}
\bE[|W f(\bX_0)|] \le  \frac{C}{N^{s/2}} \, (4r e {\bC}_\star)^{k-k_\bJ} 
\, k_\bJ! \, (k_\bJ+s)^s \,  \Big( \frac{4 r e {\bC}_\star}{\sqrt{N}} \Big)^{k_\bJ}
\sum_{I_\star =0}^{k_\bJ} \frac{N^{I_\star}}{I_\star!} \,.
\end{equation}
Since $k! \ge k_\bJ! (k-k_\bJ)!$,  the bounds \eqref{eq:w-bd} and \eqref{eq:L-decompose-words} 
will yield the stated absolute convergence of the infinite series. Specifically, fixing $T<\infty$ 
and setting $\delta=1/(16 T r e {\bC}_\star)$, we have that 
\begin{align}
\sum_{k=0}^\infty\frac{T^{k}}{k!}\bE[|L^{k}f(\bX_0)|] \leq & \frac{C}{N^{s/2}}  
\sum_{k =0}^{\infty} \sum_{k_\bJ \le k}\frac{(4 T)^k}{k!}  (4r e {\bC}_\star)^{k-k_\bJ} 
\, k_\bJ! \,(k_\bJ+s)^s \,  \Big( \frac{4 r e {\bC}_\star}{\sqrt{N}} \Big)^{k_\bJ} 
\sum_{I_\star =0}^{k_\bJ} \frac{N^{I_\star}}{I_\star!} \nonumber
\\
\leq & \frac{C}{N^{s/2}} 
\sum_{k' = k-k_\bJ = 0}^\infty \frac{ \delta^{-k'}}{k'!} \, 
\sum_{k_\bJ = 0}^\infty (k_\bJ+s)^s \,  \big( \delta \sqrt{N} \big)^{-k_\bJ} 
\sum_{I_\star =0}^{k_\bJ} \frac{N^{I_\star}}{I_\star!}
\,,
\label{eq:bd-abs-sum}
\end{align}
which is finite for any fixed $N > \delta^{-2}$, 
thereby concluding the proof.
\end{proof}

\subsection{Controlling the differences of the $k$'th order Taylor coefficients}\label{subsec:difference-in-expectations}

By 
Corollary~\ref{cor:switch-expectation-sum}, we have that
\begin{align}
\sup_{t \in [0,T]} \Big|\bE[f(\bX_{t})]-\tbE[f(\tbX_{t})]\Big|\leq & \sup_{t\in [0,T]} \sum_{k\geq0}\frac{t^{k}}{k!}\Big|\bE[L^{k}f(\bX_0)]-
\tbE[L^{k}f(\bX_0)]\Big| \nonumber \\
 \le & \sum_{k\geq0}\frac{(4T)^{k}}{k!}
\sup_{W\in\{L_{\bJ},L_{\bLambda},L_{\bh},L_{\Delta}\}^{k}}\Big|\bE[Wf(\bX_0)]-\tbE[Wf(\bX_0)]\Big| \nonumber \\
\le &
\sum_{k\geq0}\frac{(4T)^{k}}{k!}
\sup_{W\in\{L_{\bJ},L_{\bLambda},L_{\bh},L_{\Delta}\}^{k}} \sum_{\phi \in (Wf)(\bx)} \Big|\bE[\phi(\bX_0)]-\tbE[\phi(\bX_0)]\Big| 
\,,
\label{eq:series-W}
\end{align}
where the last sum is over $\phi$ appearing in the monomial decomposition of $Wf(\bx)$ per Claim~\ref{claim:monomial-structure}.
To bound the differences of expectations on the \abbr{rhs} of \eqref{eq:series-W}, we next 
control the type of monomials $\phi$ of the form~\eqref{eq:monomial-structure} in the
expansion of  $Wf$, for which we may possibly have $\bE[\phi(\bX_0)] \ne \tbE [\phi(\bX_0)]$.
\begin{lem}\label{lem:monomial-counting}
For any $k,s \ge 0$, every $s$-tuple of pairs $\balpha$, and every $W\in \{L_\bJ, L_\bLambda, L_\bh, L_\Delta\}^k$,
the monomials $\phi$ 
 in the expansion of $Wf$ in Claim~\ref{claim:monomial-structure} may have $\bE[\phi(\bX_0)] \ne \tbE[\phi(\bX_0)]$ only if 
\begin{equation}\label{def:KJ-alpha}
k_\bJ +s \ge 3  \qquad \hbox{and} \qquad 
k_\bJ \ge 2 I_\star + I^+_{\balpha,1}  \,,
\end{equation} 
where, as before, 
$I_\star = I_{\balpha \amalg \bbeta} - I_\balpha$ denotes the number of distinct elements in $\{\bbeta\}\setminus \{\boldsymbol \alpha\}$.
\end{lem}
\begin{proof}  By the independence of $\bJ,\tbJ$
and $\mu$, if $\bE[\phi(\bX_0)] \ne \tbE[\phi(\bX_0)]$ for some 
$\phi = \phi_{\bbeta,\bbeta',\bzeta,\bzeta',\bxi}$ as in \eqref{eq:monomial-structure}, then 
\begin{align*}
\bE\Big[\prod_{i=1}^{s}J_{\alpha_i}\prod_{\ell=1}^{k_{\bJ}}J_{\beta_{\ell}}\Big] \neq
\tbE\Big[\prod_{i=1}^{s}J_{\alpha_i}\prod_{\ell=1}^{k_{\bJ}}J_{\beta_{\ell}}\Big]\,,
\end{align*}
which for independent, zero-mean $(J_{ij})_{ij}$ of matching variances $\frac{1}{N} \bm= \frac{1}{N} \tbm$, requires that
simultaneously:
\begin{align}
\label{item:criterion-1} 
& \text{No pair $\alpha_\star$ appears exactly once in the 
    concatenation $\balpha \amalg \bbeta$.} \\
    \label{item:criterion-2} 
    & \text{Some $\alpha_\star$ appears more than twice in 
    the concatenation $\balpha \amalg \bbeta$.}    
\end{align}
The condition~\eqref{item:criterion-1}
 implies that each of the $I_\star$ distinct elements 
in $\{\bbeta\} \setminus \{\balpha\}$ must appear at least twice in $\{\bbeta\}$, 
to which end we need at least $2  I_\star$ applications of $L_\bJ$ to select those elements.
In addition, some other $I_{\balpha,1}$ of the $k_\bJ$ applications of $L_\bJ$ must align exactly 
with the pairs $(\alpha_{ij})$ appearing only once in $\balpha$, so necessarily 
$k_\bJ \ge 2  I_\star + I_{\balpha,1}$. Further, the condition~\eqref{item:criterion-2}
requires $k_\bJ+s \ge 3$ and when no pair appears more than twice in $\balpha$, an extra application
of $L_\bJ$ beyond the preceding $2  I_\star + I_{\balpha,1}$ is needed for producing the third appearance of 
some $\alpha_\star$, as stated in \eqref{def:KJ-alpha}.
\end{proof}

We are now able to prove that the
expectations of monomials of the form $f_{\balpha,\bgamma}(\bX_{t})$ are universal. 

\begin{proof}[\textbf{\emph{Proof of Proposition~\ref{prop:difference-in-expectations}}}]
Fixing $\balpha,\bgamma$, in view of Lemma~\ref{lem:monomial-counting}, it suffices
when bounding the \abbr{rhs} of \eqref{eq:series-W}, 
 to consider only words $W$ and
monomials $\phi$ for which \eqref{def:KJ-alpha} holds. By restricting attention only to monomials for which~\eqref{def:KJ-alpha}, 
holds, we find as in \eqref{eq:reza-added}, that for any $\balpha$ whose $I_\balpha$ distinct terms appear in 
multiplicities $(c_\ell)_{\ell \le I_\balpha}$, and every word $W$ of length $k$ such that $k_\bJ + s\ge 3$, 
\[
\big|\bE[W f(\bX_0)] - \tbE[Wf(\bX_0)]\big| \le 
2 \Cmu (r) {\bC}_\star^{s} \,  k_\bJ! \,  \frac{(4 r {\bC}_\star)^{k}}{N^{(k_\bJ+s)/2}}
\sum_{\{I_\star : k_\bJ \ge 2 I_\star + I^+_{\balpha,1}\}} \frac{N^{I_\star}}{I_\star!} \sum_{(n_\ell)_{\ell \le I_{\balpha\amalg\bbeta}}} \,
\prod_{\ell =1}^{I_{\balpha}} \frac{(n_\ell+c_\ell -1)!}{n_\ell !} \,,
\]
where, as in \eqref{eq:reza-added}, the inner sum runs over all
partitions of $k_\bJ-I_\star$ into $I_{\balpha\amalg\bbeta}$ indistinguishable
integers $n_\ell \ge 0$. Reasoning as we did leading up to~\eqref{eq:w-bd}, we find that 
\begin{align}\label{eq:w-diff-bd}
\big| \bE[W f(\bX_0)] - \tbE[W f(\bX_0)]  \big| 
& \le  \frac{2C}{N^{s/2}} \, (4r e {\bC}_\star)^{k-k_\bJ} 
\, k_\bJ! \, (k_\bJ+s)^s \,  \Big( \frac{4 r e {\bC}_\star}{\sqrt{N}} \Big)^{k_\bJ}
\sum_{\{I_\star : k_\bJ \ge 2 I_\star + I^+_{\balpha,1}\}} \frac{N^{I_\star}}{I_\star!} \,. 
\end{align}
Plugging \eqref{eq:w-diff-bd} into
\eqref{eq:series-W}, as in the derivation of \eqref{eq:bd-abs-sum}, we get for 
$\delta = 1/(16 T r e {\bC}_\star)$ and $N \ge \rho := (2/\delta)^2$,
\begin{align}
\sup_{t \in [0,T]} \Big|\bE[f(\bX_{t})]-\tbE[f(\tbX_{t})]\Big|\
& \leq \frac{2C}{N^{s/2}} 
\sum_{k' \ge 0} \frac{ \delta^{-k'}}{k'!} \, 
\sum_{I_\star \ge 0} \frac{1}{I_\star!} \sum_{k_\bJ \ge 2 I_\star + I^+_{\balpha,1}} (k_\bJ+s)^s \, \delta^{-k_\bJ}   N^{I_\star-k_\bJ/2}  \nonumber \\
& \le \bar C N^{-(s+I^+_{\balpha,1})/2} 
\sum_{I_\star \ge 0} \frac{\rho^{I_\star}}{I_\star!} 
\sum_{k_\bJ \ge 0} (k_\bJ+s)^s 2^{-k_\bJ} 
\,,
\label{eq:diff-in-exp-final-bound}
\end{align}
where $\bar C=2 C e^{-1/\delta} \rho^{I^+_{\balpha,1}/2}$. This completes the proof, as
both series on the \abbr{rhs} of \eqref{eq:diff-in-exp-final-bound} are finite and independent of $N$.
\end{proof}

\begin{rem}\label{rem:symmetric-independent}
In the case of \emph{symmetric} random matrices
$\bA$, $\tbA$ (where only the upper triangular and diagonal elements are independent),
we identify index pairs $\beta = ij$ and $\hat \beta=ji$ as being the same. We do so 
whenever considering $I_\balpha$, $I_{\balpha,1}$, $I^+_{\balpha,1}$,  $I_{\balpha \amalg \bbeta}$, $I_\star$,  and 
the multiplicities $(n_\ell)_\ell$, as well as in the restrictions \eqref{item:criterion-1}--\eqref{item:criterion-2}
imposed on the multiplicities within $\balpha \amalg \bbeta$. Once this is done, the 
only difference in our proof is to replace in \eqref{eq:bd-monomial-number}
the weight $r^k$ by $(2r)^k$. 
\end{rem}

\section{The extension to multi-time polynomial observables}\label{sec:more-general-observables}
In this section, we extend the results of Section~\ref{sec:universality-expectations} to more general observables, namely those that contain coefficients that depend on the driving martingale, and those that depend on the trajectory through multiple times, rather than just one. We then use those extensions to prove  Theorem~\ref{thm:main-expectations}.
To this end, 
fix any $l$, any $(\balpha^{(1)},\ldots,\balpha^{(l)})$ each consisting of $s_i$ pairs, any $(\bgamma^{(1)},\ldots,\bgamma^{(l)})$ each consisting of $r_i$ indices, and also fix $m$ indices $\bxi =( \xi_1,\ldots,\xi_m )$. Fix  $l$ times $0\le t_1\leq \cdots \leq t_l\leq T$ and $m$ times $0\le u_1 \leq \cdots \leq u_m\leq T$. For $f_{\balpha^{(i)},\bgamma^{(i)}}$ as in~\eqref{eq:f-balpha-gamma-def}, consider observables of the form,  
\begin{align}\label{eq:multitime-observable-form}
    g_{(\balpha^{(i)}), (\bgamma^{(i)}), \bxi}(\bt, \bu) = \Big(\prod_{i=1}^l f_{\balpha^{(i)}, \bgamma^{(i)}} (\bX_{t_i})\Big)\Big(\prod_{i=1}^m M_{\xi_i}(u_i)\Big)\,.
\end{align}
Let $\bar r = \sum_i r_i +m$ and
$\bar \balpha$ denote the concatenation $\balpha^{(1)} \amalg \cdots \amalg \balpha^{(l)}$ of length 
$\bar s := \sum_i s_i$.
\begin{prop}\label{prop:bm-difference-in-expectations}
There exist finite $C(\bar r, \bar s, m,l, T, \bC_\star,\Cmu(\bar r))$ such that 
for every $l,m$, every $(\balpha^{(i)})_{i\le l}$, $(\bgamma^{(i)})_{i \le l}$, $\bxi$, every
$\bt \in [0,T]^l$, $\bu \in [0,T]^m$ and
$g(\bt,\bu)= g_{(\balpha^{(i)}),(\bgamma^{(i)}),\bxi} (\bt, \bu)$
as in~\eqref{eq:multitime-observable-form},
\[
\big|\bE [g(\bt,\bu)] -\tbE [g(\bt,\bu)]\big| \leq C 
N^{-(\bar s + I^+_{\bar \balpha,1})/2}\,.
\]
\end{prop}

We proceed to prove Proposition~\ref{prop:bm-difference-in-expectations}, which we thereafter combine with a short combinatorial  estimate bounding the number of terms with specific values of $I^+_{\bar \balpha,1}$
to establish Theorem~\ref{thm:main-expectations}.

\subsection{Proof of Proposition~\ref{prop:bm-difference-in-expectations}}
We start with 
the case of $m=0$ to which we will reduce the case of $m>0$.  
\begin{lem}\label{lem:multitime-diff-in-expectations}
Proposition~\ref{prop:bm-difference-in-expectations} holds when $m=0$.
\end{lem}
\begin{proof} Fixing $l$, $(\balpha^{(i)})_{i\le l}$ and $(\bgamma^{(i)})_{i \le l}$, we set here
$f^{(i)}(\bx)=f_{\balpha^{(i)},\bgamma^{(i)}}(\bx)$ and 
 \begin{align}\label{eq:g-multi-time}
g (\bx^{(1)},\ldots,\bx^{(l)}) := & \prod_{i=1}^{l} f^{(i)}(\bx^{(i)}) = \prod_{i=1}^l  J_{\alpha_1^{(i)}} \cdots J_{\alpha_{r_i}^{(i)}} \, x^{(i)}_{\gamma_{1}^{(i)}} \cdots x^{(i)}_{\gamma_{s_i}^{(i)}} \,,
\end{align}
and for any $l$-tuple of times $\bt = (t_1,\ldots, t_l)\in [0,T]^l$, evaluate \eqref{eq:g-multi-time}  
on the argument $(\bX_{t_1},\ldots, \bX_{t_l})$: i.e., let
\[
g(\bt) = g_{(\balpha^{(i)}),(\bgamma^{(i)})} (\bt) = g(\bX_{t_1},\ldots, \bX_{t_l})\,.
\]
We express the expectation
$\bE_{\bB}$ with respect to the Brownian motion of $g(\bt)$, in terms of the (diffusion) 
semi-group operator as 
\begin{align*}
    \bE_\bB[g(\bt)] = \big(P_{t_1} f^{(1)} P_{t_2 - t_1}  f^{(2)} \cdots P_{t_l - t_{l-1}}  f^{(l)} \big) (\bX_0)\,.
\end{align*}
Expanding each semi-group operator  in terms of powers of the generator $L$, the above is precisely 
\begin{align*}
    \sum_{k_1\geq 0} \frac{t_1^{k_1}}{k_1!} L^{k_1}   \Big [ f^{(1)} \sum_{k_2\geq 0} 
    \frac{(t_2-t_1)^{k_2}}{k_2!} L^{k_2} \Big[ f^{(2)} & \cdots \sum_{k_l\geq 0} \frac{(t_l - t_{l-1})^{k_l}}{k_l!} L^{k_l} f^{(l)} \Big]\Big](\bX_0) \\ 
    & = \sum_{k_1,\ldots,k_l \geq 0} \Big( \prod_{i=1}^l \frac{(t_i-t_{i-1})^{k_i}}{k_i!} \Big)
    \big[L^{k_1}f^{(1)} L^{k_2} f^{(2)}  \cdots L^{k_l} f^{(l)}\big](\bX_0)\,.
\end{align*}
Taking the difference in expectations between $\bE$ and 
$\tbE$, upon justifying swapping the expectation with the infinite sum 
(as done in Section~\ref{subsec:switch-expectation-sum}), and using the fact that
\begin{align}\label{eq:factorial-inequality}
{k \choose k_1,\ldots, k_l} l^{-k} \le 
\sum_{\substack{k_1,\ldots,k_l \ge 0 \\ \sum k_i =k}} \, {k \choose k_1,\ldots, k_l} l^{-k} = 1 \,,
\end{align}
for every $k_1,k_2,\ldots,k_l$ such that $k_1+\cdots +k_l= k$, we obtain that 
\begin{align*}
   \big|\bE   [ g(\bt)]  - \tbE[g(\bt)]\big| \leq 
    & \sum_{k\geq 0}\sum_{\substack{k_1,\ldots,k_l \\ \sum_i k_i =k}} \frac{l^k T^{k}}{k!} \!\!\!\!\!\sum_{\substack{W_1,\ldots,W_l \\ W_i \in \{L_\bJ, L_\bLambda, L_\bh ,L_\Delta\}^{k_i}}} \!\!\!\!\!\!\!\!\!\!
     \big| \bE [(W_1 f^{(1)} \cdots W_l f^{(l)})(\bX_0)] -
    \tbE[(W_1 f^{(1)} \cdots W_l f^{(l)})(\bX_0)]\big|\,.
\end{align*}
The following structural property for words appearing in the above will allow us to reduce the analysis 
of multi-time observables to the combinatorial analysis of one-time observables 
$f_{\bar \balpha,\bar\bgamma}=f^{(1)} f^{(2)} \cdots f^{(l)}$,
for $\bar \balpha = \balpha^{(1)} \amalg\cdots \amalg \balpha^{(l)}$ and 
$\bar \bgamma := \bgamma^{(1)} \amalg \cdots \amalg \bgamma^{(l)}$,
which we have already completed.  
\begin{claim}\label{clm:multitime-expansion-vs-onetime-expansion}
Fix $k_1,\ldots,k_l \ge 0$ such that $\sum_{i} k_i=k$ and words 
$W_i\in \{L_\bJ, L_\bLambda, L_\bh ,L_\Delta\}^{k_i}$, $i=1,\ldots,l$,
with $k_\bJ^{i}, k_\bLambda^{i}, k_\bh^{i}, k_\Delta^{i}$, of each appearing, respectively. Then, the function
\[
(W_1 f^{(1)} W_2 f^{(2)} \cdots W_l f^{(l)}) (\bx)
\]
consists of a sum of (not necessarily distinct) monomials of the form 
\begin{align*}
    \phi(\bx) = \prod_{i=1}^{s_1} J_{\alpha_i^{(1)}}\cdots  \prod_{i=1}^{s_l} J_{\alpha_i^{(l)}} \prod_{\ell=1}^{\sum k_\bJ^i} J_{\beta_\ell} \prod_{\ell=1}^{\sum k_\bLambda^i} \Lambda_{\beta'_\ell} \prod_{\ell=1}^{\sum k_\bh^i} h_{\zeta_\ell'}\prod_{\ell =1}^{2\sum k_\Delta} \sigma_{\zeta_\ell} \prod_{\ell=1}^{\sum r_i} x_{\xi_\ell}\,.
\end{align*}
Moreover, each monomial $\phi(\bx)$ appearing in this expansion, 
must also appear in such monomial expansion of 
$W f_{\bar \balpha,\bar \bgamma}$ for  $W = W_1 \cdots W_l \in \{L_\bJ, L_\bLambda, L_\bh ,L_\Delta\}^{k}$.
\end{claim}

\begin{proof}
The structure of the monomials is evident. Every such monomial in $W_1 f^{(1)} W_2 f^{(2)} \cdots W_l f^{(l)}$
must also appear in the monomial expansion of $[W_1 \cdots W_l] f_{\bar \balpha,\bar \bgamma}$ because 
a subset of the terms in the latter are obtained by applying  the letters in $W_l$ to $f^{(l)}$, then the letters in $W_{l-1}$ to $f^{(l-1)} (W_l f^{(l)})$, and so on. Finally, observe that $W_1 \cdots W_l$ is always a word in $\{L_\bJ, L_\bLambda, L_\bh, L_\Delta\}^{k}$. 
\end{proof}

With Claim~\ref{clm:multitime-expansion-vs-onetime-expansion} in hand, we further get that 
\begin{align}\label{eq:multitime-difference-expansion}
|\bE[g(\bt)] - \tbE[g(\bt)]| & \le 
    \sum_{ k \geq 0} \sum_{\substack{k_1,\ldots,k_l \\ \sum k_i=k}}  \frac{4^k l^k T^k}{k!} \sup_{\substack{W_1,\ldots,W_l \\ W_i \in \{L_\bJ, L_\bLambda, L_\bh ,L_\Delta\}^{k_i}}} \sum_{\phi \in (W_1f ^{(1)}\cdots 
    W_l f^{(l)})(\bx)} 
    \!\!\!\!\!\! \big|\bE[\phi(\bX_0)] - \tbE[\phi(\bX_0)]\big| \nonumber \\ 
     & \leq \sum_{k \geq 0} \sum_{\substack{k_1,
     \ldots,k_l \\ \sum k_i=k}}  \frac{(4 l T)^k}{k!} \sup_{\substack{W \in \{L_\bJ, L_\bLambda, L_\bh ,L_\Delta\}^{k}}} 
     \sum_{\phi \in (W f_{\bar \balpha,\bar \bgamma})
      (\bx)} \big|\bE[\phi(\bX_0)] - \tbE[\phi(\bX_0)]\big| \nonumber \\ 
     & \leq \sum_{k \geq 0}  \frac{ (k+1)^l (4l T)^k}{k!} \sup_{W\in \{L_\bJ, L_\bLambda, L_\bh, L_\Delta\}^k } \sum_{\phi \in 
     (W f_{\bar \balpha,\bar \bgamma})(\bx)} \big|\bE[\phi(\bX_0)] - \tbE[\phi(\bX_0)]\big|\,,
\end{align}
where the sums are over the monomials $\phi$ in the decomposition of 
$W_1 f^{(1)} \cdots W_l f^{(l)}$ and that of
$W f_{\bar \balpha,\bar \bgamma}$ per Claim~\ref{clm:multitime-expansion-vs-onetime-expansion}. Note that each summand on the \abbr{rhs} of 
\eqref{eq:multitime-difference-expansion} is at most some $(k+1)^l l^k$ times the corresponding 
summand of~\eqref{eq:series-W} for the choice $f=f_{\bar \balpha,\bar \bgamma}$ for
which we have deduced the bound of~\eqref{eq:w-diff-bd}. Utilizing the latter and the elementary bound
$k+1 \le (k_\bJ+1)(k+1-k_\bJ)$, by proceeding as in the derivation of \eqref{eq:diff-in-exp-final-bound}, 
we find that for $C=C(\bar r,\bar s,\Cmu(\bar r),{\bC}_\star)$ finite, 
$\delta = 1/(16 \, l \, T {\bar r}\, e \, {\bC}_\star)$ positive and $N \ge (2/\delta)^2$,
\begin{align*}
\sup_{ \bt \in [0,T]^l} \big|\bE [g(\bt)]-\tbE[g(\bt)]\big| 
& \leq \frac{2C}{N^{\bar s/2}} 
\sum_{k' \ge 0} \frac{ \delta^{-k'}}{k'!} (k'+1)^l \, \sum_{I_\star \ge 0} \frac{1}{I_\star!} 
 \sum_{k_\bJ \ge 2 I_\star + I^+_{\bar \balpha,1}}
(k_\bJ+\bar s)^{\bar s+l} \, \delta^{-k_\bJ}  \, N^{I_\star-k_\bJ/2} 
 \nonumber \\ &
 \le \bar C N^{-(\bar s+I^+_{\bar \balpha,1})/2} 
\end{align*}
for some finite $\bar C=\bar C(l,\bar r,\bar s, T, \bC_\star,\Cmu(\bar r))$.
\end{proof}

We now add in the driving martingale observables (i.e., $m>0$) and conclude the proof of Proposition~\ref{prop:bm-difference-in-expectations}.
\begin{proof}[\textbf{\emph{Proof of Proposition~\ref{prop:bm-difference-in-expectations}}}]
We reduce the situation $m>0$ to the combinatorial calculations of 
Lemma~\ref{lem:multitime-diff-in-expectations} by utilizing the following expansion from Ito's lemma: 
\begin{align*}
  M_{\xi_i}(u) = X_{\xi_i} (u) - \int_0^u (L \, x_{\xi_i})(\bX_{\tau})d\tau\,.
\end{align*}
When expanding~\eqref{eq:multitime-observable-form} in this manner, the terms containing only 
products of $X_{\xi_i}(u_i)$ 
can be absorbed into $\bgamma$, in which case their difference in expectations has already been handled 
in Lemma~\ref{lem:multitime-diff-in-expectations}, so by linearity it suffices for us to focus on 
handling terms of the form
\begin{align*}
    h_{(\balpha^{(i)}),(\bgamma^{(i)}),\bxi}(\bt,\bu) = \Big(\prod_{i=1}^l f_{\balpha^{(i)},\bgamma^{(i)}} (\bX_{t_i}) \Big)\Big(\prod_{i=1}^m \int_0^{u_i} (Lx_{\xi_i})(\bX_{\tau_i})d\tau_i\Big) 
    = \int_0^{u_1} \cdots \int_0^{u_m} \widehat h(\bt,\btau) d\tau_1\cdots d\tau_m   \,,
\end{align*}   
where $\btau = (\tau_1,\ldots,\tau_m) \in [0,T]^m$ and where, setting 
$f^{(i)}(\bx)=f_{\balpha^{(i)},\bgamma^{(i)}}(\bx)$, 
\begin{align*}
\widehat h(\bt,\btau) := 
\prod_{i=1}^{l} f^{(i)} (\bX_{t_i}) \prod_{i=1}^m (L x_{\xi_i})(\bX_{\tau_i}) \,.
\end{align*}
Thus, fixing $l,m$, $(\balpha^{(i)}),(\bgamma^{(i)}),\bxi$ and letting $h(\bt, \bu) =     h_{(\balpha^{(i)}),(\bgamma^{(i)}),\bxi}(\bt,\bu) $ we obtain after swapping the expectation and integrals that 
\begin{align*}
     \bE\big[h(\bt,\bu)\big]  = \int_0^{u_1} \cdots \int_0^{u_m}    \bE\big[     \widehat h(\bt,\btau) \big] d\tau_1\cdots d\tau_m  
    \,,  
\end{align*}
which thereby yields the following bound on the relevant difference in expectations 
\begin{align*}
   \big|\bE [h(\bt,\bu)] - & \tbE [h(\bt,\bu)]\big| 
    \leq T^m \sup_{\btau \in [0,T]^m} \Big|  \bE\big[     \widehat h(\bt,\btau) \big]
    - \tbE\big[     \widehat h(\bt,\btau) \big]
   \Big| \,.
\end{align*}
Proceeding hereafter \abbr{wlog} to bound the difference in expectations for $\widehat h(\bt,\btau)$, 
we suppose for ease of exposition that $0 \leq t_l = \tau_0 \leq \tau_1\leq \cdots \leq \tau_m$
(the situation where the two groups intertwine is similarly analyzed with the obvious modifications). As done
in the proof of Lemma \ref{lem:multitime-diff-in-expectations},  first expressing $\bE_{\bB}$ 
in terms of the semi-group operator and then expanding that in powers of the generator $L$ we find 
that 
\begin{align*}
 &  \bE_{\bB} \big[     \widehat h(\bt,\btau) \big]
  = P_{t_1}\Big[f^{(1)} P_{t_2 - t_1} \Big[f^{(2)} \cdots P_{t_l-t_{l-1}}\Big[ f^{(l)} P_{\tau_1 - t_l} \big[Lx_{\xi_1} \cdots P_{\tau_m - \tau_{m-1}} Lx_{\xi_m}\big]\Big]\Big]\Big](\bX_0) \\
   & =  \sum_{k \geq 0} \sum_{\substack{(k_i)\geq 0, (\ell_i)\geq 1 \\ \sum k_i + \sum \ell_i = k+m}} 
\prod_{i=1}^l \frac{(t_i-t_{i-1})^{k_i}}{k_i!} \prod_{i=1}^m \frac{(\tau_i-\tau_{i-1})^{\ell_i-1}}{(\ell_i-1)!}   
    L^{k_1}\Big[f^{(1)} \cdots L^{k_l}\Big[ f^{(l)} L^{\ell_1}\big[x_{\xi_1} \cdots L^{\ell_m} x_{\xi_m}\big]\Big]\Big](\bX_0)\,.
\end{align*}

At this point, proceeding as in the derivation of \eqref{eq:multitime-difference-expansion}, up to the 
transformations
 \begin{align*}
  k \mapsto k+m = : \bar k \,, \qquad  l\mapsto l+m =: \bar l \,, \quad \mbox{and} \quad 
  (f^{(l+1)},\ldots,f^{(\bar l)}) \mapsto (x_{\xi_1},\ldots,x_{\xi_m})\,,
\end{align*}
we first use ~\eqref{eq:factorial-inequality} to get the bound
\begin{align*}
  \Big|  \bE\big[     \widehat h(\bt,\btau) \big]
    & - \tbE\big[     \widehat h(\bt,\btau) \big]
   \Big| \le \\
   & 
    \sum_{k \geq 0}\sum_{\substack{(k_i)\geq 0, (\ell_i)\geq 1 \\ \sum k_i + \sum \ell_i = \bar k}} \frac{4^{\bar k} (\bar l T)^k}{k!}\sup_{\substack{W_1,\ldots,W_{l},W'_1,\ldots,W'_m \\ W_i \in \{L_\bJ,L_\bLambda,L_\bh,L_\Delta\}^{k_i} \\ W'_i\in \{L_\bJ,L_\bLambda,L_\bh,L_\Delta\}^{\ell_i}}} 
    \sum_{\phi\in (W_1 f^{(1)} \cdots W'_m  x_{\xi_m}) (\bx)} \big|\bE [\phi(\bX_0)]-\tbE[\phi(\bX_0)]\big| \,,
\end{align*}
with the sum running over monomial decomposition of 
$(W_1 f^{(1)} \cdots W_l f^{(l)} W_1' x_{\xi_1} \cdots W'_m  x_{\xi_m}) (\bx)$. Then, utilizing again 
Claim~\ref{clm:multitime-expansion-vs-onetime-expansion}, as well as the bound $k! \ge \bar k!/(\bar k)^m$, we arrive at 
\begin{align}\label{eq:bm-difference-expansion}
 \Big|  \bE\big[     \widehat h(\bt,\btau) \big]
    & - \tbE\big[     \widehat h(\bt,\btau) \big]
   \Big| \le \nonumber \\
   &  
    \sum_{\bar k \ge m} \Big(\frac{\bar k}{\bar l T}\Big)^m \,
    \frac{(\bar k)^{\bar l} (4 \bar l T)^{\bar k}}{\bar k!} 
   \sup_{W \in \{L_\bJ,L_\bLambda,L_\bh,L_\Delta\}^{\bar k}} \sum_{\phi\in (W f_{\bar \balpha,\bar \bgamma})(\bx)} \big|\bE [\phi(\bX_0)]-\tbE[\phi(\bX_0)]\big| \,,
\end{align}
where as before $\bar \balpha = \balpha^{(1)}\amalg \cdots \amalg \balpha^{(l)}$ is
of length $\bar s = \sum_i s_i$, while $\bar \bgamma$ of length $\bar r = \sum r_i +m$ has now
the additional elements $(x_{\xi_i})_{i\le m}$. Up to this update of $\bar r$ and the immaterial
weight factor $(\bar k/(\bar l T))^m$ of its summands, the expression on the \abbr{rhs} 
of~\eqref{eq:bm-difference-expansion} is the same as that in \eqref{eq:multitime-difference-expansion}.
We thus conclude as in the proof of Lemma~\ref{lem:multitime-diff-in-expectations} that for some 
$C(l,m,\bar r, \bar s, T, \bC_\star,\Cmu(\bar r))$ all $\bt \in [0,T]^l$ and
$\bu \in [0,T]^m$, 
\[
\big|\bE [h(\bt,\bu)] -\tbE [h(\bt,\bu)]\big| \leq C 
N^{-(\bar s + I^+_{\bar \balpha,1})/2}\,.\qedhere
\]
\end{proof}

\subsection{Proof of Theorem~\ref{thm:main-expectations}.}
Fix $T,m,p$, $C_\ba$, $\ba \in \bR^{N^{m}}$ such that $\|\ba\|_\infty \le C_\ba$, and $\bt\in [0,T]^p$. For every $\ell\le m$, fix observables $\cY^{(\ell,1)},\ldots,\cY^{(\ell,p)}\in {\mathfrak F}$ and let $F(\bt)$ be as in~\eqref{eq:F-general-form} with those choices. By linearity of expectations and the 
uniform bound on $\|\ba\|_\infty$, it suffices to show that uniformly over $i_1,\ldots,i_m$, 
\begin{align}\label{eq:wts-general-expectation-universality}
\sup_{\bt \in [0,T]^p} \Big| \bE \big[ \prod_{\ell \le m} \cY_{i_\ell}^{(\ell,1)}(t_1) \cdots \cY_{i_\ell}^{(\ell,p)} (t_p) \big] - \tbE \big[ \prod_{\ell \le m} \cY_{i_\ell}^{(\ell,1)}(t_1) \cdots \cY_{i_\ell}^{(\ell,p)} (t_p) \big]\Big| \le C N^{-1/2}
\end{align}
We denote by $\bar s$ the number of $\cY$ terms appearing in the preceding product which is a coordinate
of $\bG_t$. In case $\bar s=0$, the bound \eqref{eq:wts-general-expectation-universality} follows from 
considering Proposition \ref{prop:bm-difference-in-expectations} at $\bar s =0$, in which case 
$I^+_{\bar \balpha,1}=1$. Otherwise, we expand every term in that product which is a coordinate of $\bG_t$
to obtain a sum of monomials of the form of~\eqref{eq:multitime-observable-form}. 
Each of these monomials has a sequence $\bar \balpha$ of length $\bar s$, and as a result of such expansion 
there are at most $\bar s^{\bar s} N^{I_{\bar \balpha}}$ monomials with precisely $I_{\bar \balpha}$ distinct pairs in 
the sequence $\bar \balpha$. Note that for any $\bar \balpha$,
\[
\bar s+I^+_{\bar \balpha,1} \ge 2 I_{\bar \balpha} + 1 \,.
\]
Indeed, each pair which appears once in $\bar \balpha$, is counted both in $\bar s$ and in $I_{\bar \balpha,1}$, 
all other pairs  are counted at least twice in $\bar s$, and for any $\bar \balpha$ of maximal multiplicity 
two, we have added one to  $I^+_{\bar \balpha,1}$. Consequently, the bound of Proposition 
\ref{prop:bm-difference-in-expectations} on the difference in expectation for each of these $\bar s^{\bar s} N^{I_{\bar \balpha}}$ many monomials
is at most $C N^{-I_{\bar \balpha}-1/2}$ for some constant $C(T, m,p, \bC_{\star}, \Cmu)$. From this, the bound 
\eqref{eq:wts-general-expectation-universality} immediately follows upon enumerating over 
the at most $\bar s$ many choices for $I_{\bar \balpha}$. \qed

\section{Concentration for quadratic observables: Proof of Theorem \ref{thm:conc}}
\label{sec:concentration}

Assuming henceforth that $\bM_t$ is a scaled Brownian motion (i.e., that $\sigma_{ij}$ are identically 
zero for $i\ne 0$), our goal is to prove Theorem \ref{thm:conc} about the uniform over $\bt \in [0,T]^2$
concentration property of the quadratic observable of~\eqref{eq:possible-observables}, 
\[   
F(\bt) = F_{\cY, \cY', \ba}(\bX_{t_1},\bX_{t_2}) = \frac 1{N} \sum_{i} a_i \cY_i(t_1) \cY_i' (t_2)\,,
\]
(for uniformly bounded non-random $\ba = (a_i)_{i}$ 
and $\cY, \cY'$ in the collection ${\mathfrak F} = \{\one_t, \bX_t, \bG_t, \bM_t\}$ of 
\eqref{eq:building-blocks}). To this end, we introduce in Subsection~\ref{subsec:localization} high 
probability localizing sets $\cL_{N,R}$ on which various norms of $\bX_t$ (and 
our observables $F(\bt)$), are uniformly bounded.  Subsection~\ref{subsec:lipschitz} shows that 
on $\cL_{N,R}$, such $F(\bt)$ are $O(N^{-1/2})$-Lipschitz in a mixed $\ell^2$-norm. Combining
these facts we prove Theorem \ref{thm:conc} in Subsection \ref{sec:trajectory-universality}.

\subsection{Localizing the process}\label{subsec:localization}
Denote the $2$-to-$2$ matrix norm by
\[
\|\bJ\|_{2\to 2}: = \sup_{\bx:\|\bx\|=1} \|\bJ \bx\| = \| \bJ^T \|_{2 \to 2} = 
\sup_{\bx:\|\bx\|=1} \big(\sum_{i\leq N} G_i(\bx)^2\big)^{1/2}\,,
\] 
and for each constant $R$ consider the following localization subset of 
 $\cE_N := \bR^N \times \bR^{N^2} \times C([0,T],\bR^N)$,
\begin{equation}\label{def:L-N-R}
\cL_{N,R} := \Big\{(\bX_0 ,\bJ, \bM) \in \cE_N : 
\|\bX_0\|^2 + N \|\bJ\|_{2\to 2}^2 + \sup_{t\in [0,T]} \|\bM_t\|^2 \leq R N \Big\}\,,
\end{equation}
We begin by bounding the probability that $(\bX_0,\bJ,\bM)\notin \cL_{N,R}$.
\begin{lem}\label{lem:L-N-R-complement}
There exists  $C= C(T, C_\mu, C_\bA,C_\bsigma)>0$ and $R_0(T,C_\mu,C_\bA,C_\bsigma)<\infty$, such that for every $R \ge R_0$ 
if $\mu,\bP_{\bA}$ satisfy Hypotheses~\ref{hyp:1}--\ref{hyp:2}, then 
\begin{align*}
    \bP \big(\cL_{N,R}^c\big) \leq \exp (- \sqrt {R N}/C)\,. 
\end{align*}
\end{lem}

\begin{proof} We bound $\cL_{N,R}^c$ by the union of the events where each of the three norms is greater than $\sqrt{RN/3}$. First, since $\bM_t$ is a Brownian motion (scaled by $(\sigma_{0j})_j$), 
by Doob's maximal inequality for the sub-martingale $\exp(\delta \|\bM_t\|^2)$, 
we have for some $C(C_\bsigma)>0$ any $R \ge T R_0(C_\bsigma)$ and all $N$, 
\begin{align}\label{eq:bd-M-sup}
    \bP_{\bB} \Big(\sup_{t\in [0,T]} \|\bM_t\| >\sqrt{RN/3}\Big)\leq \exp (- R N/(C T))\,.
\end{align}
Next, since $\mu$ satisfies Hypotheses~\ref{hyp:1}--\ref{hyp:2}, the independent
$X_i(0)$ have uniform (in $i$ and $N$), second moments and exponential tails.
Hence, applying ~\cite[Theorem 3]{Negaev}
for the centered sum of i.i.d.\ variables that stochastically dominate 
$X_i^2(0)$, we have for some $C(C_\mu)>0$, 
any $R \ge R_0(C_\mu)$ 
and all $N$,
\begin{align*}
    \mu\Big(\|\bX_0\|^2 >RN/3\Big)\leq \exp( -\sqrt{RN}/C)\,.
\end{align*}
It thus remains only to show that when $\bP_{\bA}$ satisfies Hypothesis~\ref{hyp:2}, 
we have for some $C(C_\bA)>0$ any $R \ge R_0(C_\bA)$ and all  $N$,  
\begin{align}\label{eq:matrix-norm-bound}
    \bP_{\bA} (\|\bA\|_{2\to 2} >\sqrt{RN/3}) \leq  \exp( - \sqrt{RN}/C)\,.
\end{align}
To this end, recall~\cite[Theorem 2]{Latala} that there exists a universal constant $C$ such that for any matrix $\bA$ with independent, 
zero-mean entries of second moments $m_{ij}$ and fourth moments $b_{ij}$,  
\begin{align*}
    \bE_{\bA} [\|\bA\|_{2\to 2}]\leq C\Big(\max_{i \le N} \Big(\sum_{j\le N} m_{ij}\Big)^{1/2}+ \max_{j \le N}\Big(\sum_{i \le N} m_{ij} \Big)^{1/2} + \Big(\sum_{1\le i,j\le N} b_{ij}\Big)^{1/4}\Big)\,.
\end{align*}
For $\bP_{\bA}$ satisfying Hypothesis~\ref{hyp:1}, $b_{ij}$ and $m_{ij}$ are bounded 
uniformly in $i,j$ and $N$ (see \eqref{eq:unif-mom-bd}). Hence, in the case where $\bA$ is composed of independent entries, for some $C(C_\bA)$ finite and all $N$,
\begin{align}\label{eq:bd-2-2-mean}
     \bE_{\bA} [\|\bA\|_{2\to 2}] \leq C\sqrt{N} \,.
\end{align} 
Likewise, representing a symmetric $\bA$ as $\bA = \bA^+ + \bA^-$, with $\bA^+$ the upper triangle
(including the diagonal) part of $\bA$ and $\bA^-$ its lower triangle part, \cite[Theorem 2]{Latala} holds for 
the matrices $\bA^-$ and $\bA^+$ of zero-mean, independent entries (with uniformly bounded
forth moments). Thus, \eqref{eq:bd-2-2-mean} holds also in this case up to a factor of 2. Thanks to \eqref{eq:bd-2-2-mean}, 
if  $\sqrt{R} \ge 4 C$ then
\[
\bP_{\bA} \big(\|\bA\|_{2\to 2} > \sqrt{RN/3} \big) \leq 
\bP_{\bA} \big(|\,\|\bA\|_{2\to 2} - \bE_{\bA}[\|\bA\|_{2\to 2}]\,| > \sqrt{RN}/4 \big) \,.
\]
Recall that $\|\bA\|_{2\to 2}$, which is the largest singular value of $\bA$, 
is $1$-Lipschitz in its entries (endowed with the Euclidean norm, on $\bA^+$ 
when $\bA$ assumed symmetric). Indeed, this follows by combining the triangle inequality $|\|\bA\|_{2\to2}  - \|\bB\|_{2\to2}|\le \|\bA - \bB\|_{2\to2}$ with the domination of the operator norm by the Frobenius norm, $\|\bA - \bB\|_{2\to 2} \le \|\bA - \bB\|_{F}$.
Hypothesis~\ref{hyp:2} for $\bP_{\bA}$ thus yields the bound \eqref{eq:matrix-norm-bound}.
\end{proof}

We further have on the sets $\cL_{N,R}$ the following localization for
both $(\bX_t)_{t\in [0,T]}$ and $(\bG_t)_{t \in [0,T]}$.
\begin{prop}\label{prop:X-t-localization}
There exists $R_0(T,\bC_\star)$ and $C_0(\bC_\star)$ such that if $R \ge R_0$, 
and  $(\bX_0, \bJ,\bM)\in \cL_{N,R}$, then 
\begin{align}\label{eq:x-t-bound-localized}
\frac{1}{\sqrt{N}} \sup_{t\in [0,T]} \{  \|\bX_t\| \} \leq e^{C_0 \sqrt R T}\,, \qquad 
\frac{1}{\sqrt{N}} \sup_{t\in [0,T]} \{ \|\bG_t\| \} \leq e^{C_0 \sqrt R T}\,. 
\end{align}
In addition, for every $\ba$ such that $\|\ba\|_\infty \le C_\ba$ (uniformly over $N$) and every $\cY, \cY' \in {\mathfrak F}$, if $F(\bt)$ is as in~\eqref{eq:possible-observables}, we have for all $k \ge 1$,   
\begin{align}\label{eq:x-t-moment-bound}
 \limsup_{N\to\infty} \bE\big[ \big( \frac{1}{\sqrt{N}} \sup_{t\in [0,T]}  & \|\bX_t\| \big)^k \big] <\infty\,,
  \quad 
\limsup_{N\to\infty} \bE \big[ \big(\frac{1}{\sqrt{N}} \sup_{t\in [0,T]} \|\bG_t\| \big)^k\big] <\infty\,, \\
& \limsup_{N\to\infty} \bE\big[ \sup_{\bt\in [0,T]^2}  |F(\bt)|^k \big] <\infty \,.
\label{eq:bd-F-sup}
\end{align}
\end{prop}
\begin{proof}
Setting $e_N(t) = \frac{1}{\sqrt N}\|\bX_t\|$, we get upon expanding \eqref{eq:sde-def}, that 
\begin{align*}
    (e_N(t))^2 &\leq \frac{1}{N} \sum_{j \le N} |X_j(t)|\Big(|X_j(0)|+|M_j(t)|+ \int_0^t |h_j| ds
    +\int_0^t |G_j(\bX_s)|ds+\int_0^t |\Lambda_{j} (\bX_s)| ds \Big) \\
    &=: I_1+I_2+I_3+I_4+I_5 \,.
\end{align*}
From the definition of the $2$-to-$2$ norm, evidently
\begin{equation}\label{eq:G-to-X}
\|\bG_s\| = 
\sqrt{ \sum_{j \le N} G_j(\bX_s)^2 }
\leq \|\bJ\|_{2\to 2} \|\bX_s\|\,, \qquad 
\sqrt{ \sum_{j \le N} \Lambda_{j} (\bX_s)^2 }
 \leq \|\bLambda\|_{2\to 2} \|\bX_s\| 
  \,.
\end{equation}
Hence, by Cauchy--Schwarz, 
\begin{align*}
    I_1 &
    \leq e_N(t) \frac{1}{\sqrt N}\|\bX_0\|\,, \qquad 
    I_2 
    \leq e_N(t) \frac{1}{\sqrt N}\|\bM_t\|\,, \qquad
    I_3
    \leq e_N(t) \, C_\bh \, T \,, \\
    I_4 
    & 
    \leq e_N(t) \frac{1}{\sqrt{N}} \int_0^t \|\bG_s\|ds \leq e_N(t) \|\bJ\|_{2\to 2} \int_0^t e_N(s)ds\,,\\
    I_5 
     & 
    \leq e_N(t) \frac{1}{\sqrt{N}} \int_0^t \Big(\sum_{j\le N} |\Lambda_j (\bX_s)|^2\Big)^{1/2} ds
     \leq e_N(t) C_\bLambda \int_0^t e_N(s)ds\,,
    \end{align*}
where in the last inequality we rely on our assumption
that
$\|\bLambda\|_{1\to 1} \le C_\bLambda$ and 
$\|\bLambda\|_{\infty \to \infty} \le C_\bLambda$, to deduce that $\|\bLambda\|_{2 \to 2} \le C_\bLambda$.
Combining these bounds on $(I_i)_{i\le 5}$, and dividing out by $e_N(t)$, we see that 
\begin{align*}
    e_N(t)\leq \frac 1{\sqrt N} \Big[\|\bX_0\| + \|\bM_t\|\Big]+  C_\bh T +
    (\|\bJ\|_{2\to2} + C_\bLambda) 
    \int_0^t  e_N(s)ds\,.
\end{align*}
By Gronwall's inequality, using the localization to $\cL_{N,R}$, it then follows that for any $t \in [0,T]$,
\begin{align*}
    e_N(t) \leq (\sqrt R + C_\bh T) \exp\big((\sqrt R+C_\bLambda) t\big),
\end{align*}
yielding the \abbr{lhs} of \eqref{eq:x-t-bound-localized} as soon as
$R \ge R_0(T,{\bC}_\star) \ge 1$. From the \abbr{lhs} of \eqref{eq:G-to-X} we know that 
$\|\bG_t\| \le \sqrt{R} \, \|\bX_t\|$ throughout $\cL_{N,R}$, hence after suitably increasing 
$C_0$ and $R_0$, the \abbr{rhs} of \eqref{eq:x-t-bound-localized} holds as well.

To deduce the uniformly bounded moment estimate of \eqref{eq:x-t-moment-bound} for 
$\bX_t$, recall first from the \abbr{lhs} of \eqref{eq:x-t-bound-localized} that 
\[
Z_{N,\bX}^k := \big( \sup_{t \leq T} e_N(t) \big)^k \le e^{C_0 \sqrt{R} T k} =: f(R)\,, \qquad \forall R \ge R_0 \,,
\; (\bX_0,\bJ,\bM) \in \cL_{N,R} \,.
\]
Combining the latter bound with that of Lemma~\ref{lem:L-N-R-complement}, we arrive at
\begin{align}\label{eq:int-ZN}
\bE [Z_{N,\bX}^k] 
&= \int_0^\infty f'(R) \bP(Z_{N,\bX}^k > f(R)) dR \\
& \le f(R_0) + \int_{R_0}^\infty f'(R) \bP(\cL_{N,R}^c) d  R 
\le f(R_0) + \int_{R_0}^{\infty} f'(R) e^{-\sqrt{RN}/C} dR \,.
\nonumber
\end{align}
The \abbr{rhs} decreases in $N$ and as $f'(R)=(C_0 T k)/(2 \sqrt{R}) f(R)$, it is finite for
$\sqrt{N}/C > C_0 T k$, yielding the \abbr{lhs} of \eqref{eq:x-t-moment-bound}. The \abbr{rhs} of
\eqref{eq:x-t-moment-bound} follows by applying the same reasoning to 
$Z_{N,\bG}^k =\big( N^{-1/2} \sup_{t \in [0,T]} \|\bG_t\| \big)^k$ while utilizing the \abbr{rhs}
of \eqref{eq:x-t-bound-localized}.

Turning to \eqref{eq:bd-F-sup}, note that for any $k \ge 1$ and $F(\bt)$ of 
\eqref{eq:possible-observables} with $\|\ba\|_\infty \le C_\ba$,  by Cauchy--Schwarz, 
 \[
|F(\bt)|^k \le C_\ba^k \, \sqrt{Z_{N,\cY}^{2k}} \sqrt{Z_{N,\cY'}^{2k}} \,,   \quad \mbox{ where } \quad 
Z_{N,\cY}^{2k} := \big( \frac{1}{\sqrt{N}} \sup_{t \in [0,T]} \{ \|\cY(t)\| \} \big)^{2k}\,. 
\]
Thus, yet another application of  Cauchy--Schwarz results with 
\[
 \bE\big[  \sup_{\bt\in [0,T]^2} |F(\bt)|^k \big] \le C_\ba^k \, \sqrt{ \bE[Z_{N,\cY}^{2k}]} \sqrt{\bE[Z_{N,\cY'}^{2k}]} \le C_\ba^k \max_{\cY\in {\mathfrak F}} \bE[Z_{N,\cY}^{2k}] \,.
\]
If $\cY$ is $\one$, this latter expectation is simply 1. If $\cY$ is $\bM$, using the tail bound of~\eqref{eq:bd-M-sup} in combination with~\eqref{eq:int-ZN} (now for $f(R) = (R/3)^k$), the latter expectation is uniformly bounded in $N$. Lastly if $\cY$ is from $\{\bX,\bG\}$, the expectation above is uniformly bounded in $N$ by~\eqref{eq:x-t-moment-bound}. Combining these yields the desired~\eqref{eq:bd-F-sup}.
\end{proof}

\subsection{A Lipschitz estimate on quadratic observables}\label{subsec:lipschitz}
Our next proposition shows that on 
$\cL_{N,R}$ all $F(\bt)$ of the form~\eqref{eq:possible-observables} are $O(N^{-1/2})$-Lipschitz in the $(\bX_0, \bJ, \bM)$ endowed with the following mixed $2$-norm
on $\cE_N$,
\begin{align}\label{def:BDG-norm}
    \|(\bX_0, \bJ, \bM)\|_{\textsc{mix}}^2 := \|\bX_0\|^2 + N \sum_{1\le i ,j\le N} J_{ij}^2  + \sup_{t\in [0,T]} \|\bM_t\|^2 
\end{align}
(which is taken from \cite[Hypothesis 1.1]{BADG06}).
\begin{prop}\label{prop:main-lipschitz}
Fixing $\ba$ such that $\|\ba\|_\infty \le C_\ba$ and $\cY, \cY'\in {\mathfrak F}$, denote by $F(\bt; ( \bX_0, \bJ,\bM))$ the observable in~\eqref{eq:possible-observables} evaluated on the trajectory $\bX_t$ constructed out of the triplet $(\bX_0, \bJ, \bM)$. There exist $R_0(T,C_\ba,  \bC_\star)$ and $C(T, C_\ba, \bC_{\star})$ such that for any $R\ge R_0$ all $N$ and $(\bX_0,\bJ,\bM), (\bX'_0,\bJ',\bM')$ in $\cL_{N,R}$
\[
\sup_{\bt\in [0,T]^2}|F (\bt  ; (\bX_0, \bJ,\bM))- F (\bt  ; (\bX'_0, \bJ',\bM'))| 
\leq \frac{Ce^{ C\sqrt R}}{\sqrt N}\|(\bX_0, \bJ, \bM)-(\bX_0',\bJ',\bM')\|_{\textsc{mix}}\,.
\]
\end{prop}

The key to Proposition~\ref{prop:main-lipschitz} is to show that
$\bX_t$ is $O(1)$-Lipschitz on $\cL_{N,R}$ endowed with $\|\cdot \|_{\textsc{mix}}$. 
Specifically, denoting by $\bX_t(\bX_0, \bJ, \bM)$ the solution to~\eqref{eq:sde-def}, constructed from the triplet $(\bX_0, \bJ, \bM)$ and $\bX_t'(\bX_0, \bJ, \bM)$ the solution constructed from the triplet $(\bX_0', \bJ', \bM')$, our next lemma establishes a uniform over $\cL_{N,R}$ Lipschitz bound on $\|\bX_t-\bX'_t\|$.
\begin{lem}\label{lem:X-t-Lipschitz}
 There exist $R_0(T, \bC_\star),C(T,\bC_\star)$ such that for all $R\geq R_0$ and  $(\bX_0, \bJ, \bM),( \bX'_0 ,  {\bJ}',  \bM') \in \cL_{N,R}$, 
\[
  \sup_{t\in [0,T]} \big\|\bX_t (\bX_0, \bJ, \bM)  - \bX'_t ( \bX'_0, \bJ',\bM')\big\|  \leq  
  e^{ C\sqrt R} \big\|(\bX_0, \bJ, \bM)-( \bX'_0 ,  {\bJ}',  \bM')\big\|_{\textsc{mix}}\,.
\]
\end{lem}
\begin{proof} Following the strategy of proof of \cite[Lemma 2.6]{BADG06}, we let 
\begin{align*}
    e_N(t) : = \frac 1{\sqrt N} \|\bX_t(\bX_0 ,\bJ, \bM) - \bX_t'(\bX_0',\bJ', \bM')\|\,,
\end{align*}
and expanding over $j \le N$, we have by the definition of the solution $\bX_t$ for the \abbr{sds}
\eqref{eq:sde-def}--\eqref{eq:martingale-def}, that
\begin{align*}
e_N(t)^2  \le & \frac{1}{N}\sum_{j\le N}  |X_j(t) - X'_j(t)|\Big(  |X_j(0)- X'_j(0)|+ |M_j(t) - M'_j(t)| 
+ \int_0^t |\Lambda_j(\bX_s)  - \Lambda_j(\bX'_s)|ds \\
 & \qquad\qquad\qquad\qquad\qquad\qquad
+ \int_0^t |G_j(\bX_s) - G_j(\bX'_s)| ds 
 + \int_0^t|G_j(\bX'_s) - G'_j(\bX'_s)|ds 
  \Big)   \\
  & = : I_1 + I_2 + I_3 + I_4 + I_5 \,,
\end{align*}
where $\bG'(\cdot)$ is defined as $\bG(\cdot)$ but constructed using $\bJ'$ instead of $\bJ$. 
By Cauchy--Schwarz, 
\begin{align*}
I_1 \le e_N(t) \frac{1}{\sqrt N}\|\bX_0 - \bX'_0\|\,, \qquad \mbox{and} \qquad I_2 \le e_N(t)\frac{1}{\sqrt N} \|\bM_t - \bM'_t\|\,.
\end{align*}
Recalling~\eqref{eq:G-to-X}, we similarly find that  
\begin{align*}
I_3 & \le e_N(t) \frac{1}{\sqrt N} \int_{0}^{t} \Big(\sum_{j\le N} | \Lambda_j(\bX_s - \bX'_s)|^2\Big)^{1/2} ds \le e_N(t) C_\bLambda \int_0^t e_N(s)ds\,.
\end{align*}
Turning to the  terms involving $\bG(\cdot)$ or $\bG'(\cdot)$, observe first that 
\begin{align}
    \|\bG(\bX_t) - \bG(\bX'_t) \| &  \leq \|\bJ\|_{2\to 2} \|\bX_t - \bX'_t\|]\,,\qquad \mbox{and}\qquad \|\bG(\bX_t) - \bG'(\bX_t)\| \leq    \|\bJ - \bJ'\|_{2 \to 2}  \|\bX_t\| \,.  \label{eq:G-Lipschitz}
\end{align}
Using the localization to $\cL_{N,R}$, we thus find that 
\begin{align*}
I_4  \le e_N(t) \frac{1}{\sqrt N} \int_0^t \|\bG(\bX_s) - \bG(\bX'_s)\| ds &\le e_N(t)\|\bJ\|_{2\to 2}
\int_0^t   e_N(s)ds 
 \le e_N(t) \sqrt R \int_{0}^t e_N(s) \,, \\
I_5  \le e_N(t) \frac{1}{\sqrt N} \int_0^t \|\bG(\bX'_s) - \bG'(\bX'_s)\| ds 
&\le e_N(t) \, \|\bJ - \bJ'\|_{2\to 2} \, \frac{1}{\sqrt N} \int_0^t  \|\bX'_s\| ds \\
& \le e_N(t) \, \|\bJ - \bJ'\|_{2\to 2} \, T e^{C_0 \sqrt R T} \,,
\end{align*}
where in the last inequality we further assumed $R \ge R_0(T,{\bC}_\star)$, utilizing 
the \abbr{lhs} of \eqref{eq:x-t-bound-localized}. Further increasing $R_0$ such that 
$T e^{C_0 \sqrt R_0 T} \ge 1$, upon combining the bounds on $(I_i)_{i\le 5}$, and 
dividing out by $e_N(t)$, we see that 
\begin{align*}
    e_N(t)\leq \frac {T e^{C_0 \sqrt R T}}
    {\sqrt N} \Big[\|\bX_0 - \bX'_0\| + \sqrt{N}  
    \|\bJ - \bJ'\|_{2 \to 2} + \sup_{t\le T} \|\bM_t - \bM_t'\|\Big] 
    + \Big[ C_\bLambda  +  \sqrt R \Big]\int_0^t e_N(s)ds\,. 
\end{align*}
Recall that $\|\bJ\|^2_{2 \to 2} \le \sum_{ij} J_{ij}^2$, so by Gronwall's inequality,  
there exist $C(T,\bC_\star)$, such that
\begin{align*}
    e_N(t) \leq \frac{e^{C \sqrt{R}}}{\sqrt N} 
    \|(\bX_0 , \bJ,\bM) - (\bX'_0 ,\bJ',\bM')\|_{\textsc{mix}}\,,
\end{align*}
for any $R\ge R_0$, every $N$ and all $t\in [0,T]$, as claimed. 
\end{proof}

\begin{proof}[\textbf{\emph{Proof of Proposition~\ref{prop:main-lipschitz}}}]
Fix $\cY^1, \cY^2 \in {\mathfrak F}$, $\ba$ such that $\|\ba\|_\infty\leq C_\ba$ and $\bt = (t_1,t_2)\in [0,T]^2$.
Equipped with Lemma \ref{lem:X-t-Lipschitz} and \eqref{eq:G-Lipschitz} it remains to 
establish a Lipschitz control on differences of $F(\bt; (\bX_0,\bJ,\bM))$
in terms of differences of $\|\bG_t\|$, $\|\bX_t\|$ and $\|\bM_t\|$ corresponding to 
any pair of triplets $(\bX_0,\bJ,\bM)$ and $(\bX'_0,\bJ', \bM')$ in $\cL_{N,R}$. To this end, 
we start with the following bound on differences of $F(\bt; \cdot)$:  
\begin{align*}
     \big| F(\bt; (\bX_0, \bJ, \bM)) - F(\bt; (\bX'_0, \bJ' , \bM'))\big| & 
 \le   \frac{C_\ba}{N} \sum_{i\le N} \big|\cY^1_i(\bX_{t_1}) - \cY^1_i(\bX'_{t_1}) \big|  \big|\cY^2_i(\bX_{t_2})\big|  \\ 
    & \qquad  + \frac {C_\ba} {N} \sum_{i\le N} \big| \cY^1_i(\bX'_{t_1}) \big|\big|\cY^2_i(\bX_{t_2})- \cY^2_i(\bX'_{t_2})\big|\,.  
\end{align*}
Since the two terms on the \abbr{RHS} can be bounded symmetrically, \abbr{wlog} we focus on 
the first one, which by Cauchy--Schwarz, is at most 
\begin{align}\label{eq:bd-F-diff}
C_\ba \sup_{\cY \in {\mathfrak F}, t \in [0,T]} \big\{ \frac{1}{\sqrt{N}}\|\cY(\bX_{t})- \cY(\bX'_{t})\| \big\}
 \sup_{\cY\in {\mathfrak F}, t \in [0,T]}  \big\{ \frac{1}{\sqrt{N}} \|\cY(\bX_{t})\| \big\} \,,
\end{align}
where, as before, $\bX_t'$ is constructed out of the triplet $(\bX'_0 , \bJ', \bM')$. 
Now recall from $(\bX_0, \bJ, \bM)\in \cL_{N,R}$ and Proposition~\ref{prop:X-t-localization}, that 
the right-most term in \eqref{eq:bd-F-diff} is at most $\exp(C_0 \sqrt R T)$ for all $R \ge R_0$, in which 
case by the preceding 
\begin{align}\label{eq:bd-F-diff-LNR}
    \sup_{\bt \in [0,T]^2}\big| F(\bt; (\bX_0, \bJ, \bM)) - F(\bt; (\bX_0', \bJ ' , \bM'))\big| \leq 
    \frac{2 C_\ba e^{C_0 \sqrt R T}}{\sqrt N}\sup_{\cY \in {\mathfrak F}, t\in [0,T]} \|\cY (\bX_t) - \cY(\bX'_t)\|\,.
\end{align}
Recall Lemma~\ref{lem:X-t-Lipschitz} and~\eqref{eq:G-Lipschitz}, to deduce that for some 
$C(T,\bC_\star)>0$, every $R\ge R_0$, and all $(\bX_0, \bJ, \bM)$, we have
$(\bX'_0, \bJ', \bM')\in \cL_{N,R}$, 
\begin{align*}
    \sup_{\cY \in {\mathfrak F}, t\in [0,T]} \|\cY (\bX_t) - \cY(\bX'_t)\| \leq 
    \sqrt{R} e^{C\sqrt R} \|( \bX_0 , \bJ, \bM) - (\bX'_0,\bJ', \bM')\|_{\textsc{mix}}\,.
\end{align*}
Putting these all together, we deduce that there exists some other  $R_0 (T, \bC_\star)$ and $C(T,C_\ba,\bC_\star)$,
such that for all $R\ge R_0(T, \bC_{\star})$,
\begin{align*}
  \sup_{\substack{(\bX_0, \bJ, \bM),(\bX'_0, \bJ', \bM')\in \cL_{N,R} \\ \bt\in [0,T]^2}}
 \!\! \big| F(\bt; 
   (\bX_0, \, \bJ, \bM)) -   F(\bt; (\bX'_0, \bJ ' , \bM'))\big|  
  \leq \frac{Ce^{C \sqrt R}}{\sqrt N} \|( \bX_0 , \bJ, \bM) - (\bX'_0,\bJ', \bM')\|_{\textsc{mix}}\,.
  \quad \quad \qedhere
\end{align*}
\end{proof}

We conclude this subsection by combining the respective exponential concentrations of 
Lipschitz functions due to $\mu$, $\bP_\bA$ and~$\bP_\bB$. 
\begin{lem}\label{lem:lifting-Lipschitz}
Suppose that $\mu, \bP_{\bA}$ satisfy Hypothesis~\ref{hyp:2}. Then $\bP =\mu \otimes \bP_\bA \otimes \bP_\bB$ satisfies exponential concentration of Lipschitz functions with respect to $(\cE_N, \|\cdot \|_{\textsc {mix}})$. 
\end{lem}

\begin{proof}
Fix any function $f$ that is 1-Lipschitz on $(\cE_N, \|\cdot\|_{\textsc{mix}})$. Let us expand 
\begin{align*}
    f(\bX_0 ,\bJ, \bM) - \bE[f(\bX_0 , \bJ, \bM)] & = (f(\bX_0 , \bJ, \bM) - \bE_{\bB}[f(\bX_0,\bJ,\bM)]) + (\bE_{\bB}\big[f(\bX_0,\bJ,\bM)]- \bE_{\bJ,\bB} [f(\bX_0, \bJ,\bM)]) \\ 
    & \quad +  (\bE_{\bJ, \bB} [f(\bX_0,\bJ,\bM)] - \bE[f(\bX_0, \bJ, \bM)])\,,
\end{align*}
where the subscripts of the expectations indicate which random variables the expectation is taken over. 
Call the above three differences $I_{\bM}, I_{\bJ}$ and $I_{\bX_0}$ say. For every $\bX_0,\bJ$ fixed, $ f({\bX_0, \bJ,\bM})$ is 1-Lipschitz in $\bM \in C([0,T], \bR^N)$ endowed with the norm $\sup_{t\leq T} \|\cdot\|$. As such, from the exponential concentration of Lipschitz functions satisfied by $\bP_{\bB}$ with respect to $C([0,T],\bR^N)$ endowed with $\sup_{t\leq T} \|\cdot\|$ (see e.g., the discussion around ~\cite[Hypothesis~1.1]{BADG06}), there exists $C= C(C_\bsigma)>0$ such that for every $r>0$, 
\begin{align*}
  \sup_{\bX_0,\bJ} \bP_{\bB} \big( |I_\bM| >r/3\big) &  \leq Ce^{ - r/C} \,.
\end{align*}
Similarly, we have that for every fixed $\bX_0$, $\bE_\bB[f(\bX_0,\bJ,\bM)]$ is 1-Lipschitz in $\bJ$ endowed with its rescaled Frobenius norm $\sum_{i,j} (\sqrt N J_{ij})^2$, and finally, $\bE_{\bJ,\bB} [ f(\bX_0,\bJ,\bM)]$ is 1-Lipschitz in $\bX_0$ endowed with its $\ell^2$ norm.  Altogether, expanding
\begin{align*}
    \bP( |f(\bX_0, \bJ, \bM) - & \,\bE[f(\bX_0, \bJ, \bM)]| >r) \\
    &  \leq \bE \big[ \bP_{\bB} (|I_\bM|>r/3 \mid \bX_0,\bJ)\big]+ \bE\big[\bP_{\bJ}\big(|I_{\bJ}|>r/3 \mid \bX_0\big)\big]  + \mu \big(|I_{\bX_0}|>r/3\big) 
\end{align*}
we see that the exponential concentrations for 1-Lipschitz functions of $\mu, \bP_\bA$ and $\bP_{\bB}$ lift to exponential concentration of $\bP$ for functions that are 1-Lipschitz in the triplet $(\bX_0, \bJ,\bM)$ on $(\cE_N, \|\cdot\|_{\textsc{mix}})$. 
\end{proof}

\subsection{Proof of Theorem \ref{thm:conc}}\label{sec:trajectory-universality}
We first prove a concentration estimate for $F$ at a fixed pair of times $\bt \in [0,T]^2$, before extending this to the full trajectory $(F(\bt))_{\bt\in [0,T]^2}$ by bounding the modulus of continuity of $F$. 

\begin{prop}\label{prop:main-concentration}
Suppose $\mu$, $\bP_\bA$ satisfy Hypotheses~\ref{hyp:1}--\ref{hyp:2}. There exist
$C(T,C_\ba, \bC_\star, \Cmu)$ large, such that for any
$F$ as in~\eqref{eq:possible-observables} with  $\|\ba\|_\infty \le C_\ba$, $\cY, \cY'\in {\mathfrak F}$,
all $\bt \in [0,T]^2$, $\lambda>0$ and $N \ge N_0(T,C_\ba,\bC_\star,\Cmu)$,
\begin{align}\label{eq:conc-fixed-bt}
    \bP\big(|F(\bt) - \bE[F(\bt)]| > \lambda\big)\leq q_N(\lambda):= {\begin{cases} C e^{ - \lambda \sqrt N/C} + \lambda^{-1} e^{ - \sqrt N /C} \,, & \lambda\leq C \\ e^{- (\log \lambda) \sqrt{ N}/C} \,, & \lambda>C\end{cases}}\,.
\end{align}
\end{prop}
\begin{proof} In proving \cite[Lemma 2.5]{BADG06} it is shown, using a Lipschitz extension, that if 
$\bP$ satisfies exponential concentration for Lipschitz functions as in \eqref{eq:exp-conc}
and $V$ is an $A$-Lipschitz function on a set $\cL$ on which $|V|$ is uniformly bounded by $K$, then for some universal constant $C>0$ and every 
$\lambda>0$,
\begin{equation}\label{eq:lem25-BDG}
\bP(|V-\bE [V]| \ge \lambda) \le C e^{- \lambda/(2 A C)} + \bP(\cL^c)+ \frac{2}{\lambda} 
(\sqrt{\bE [V^2]} + K) \sqrt{\bP(\cL^c)} \,. 
\end{equation}
Recall from Lemma \ref{lem:lifting-Lipschitz} that $\bP= \mu \otimes \bP_\bA \otimes \bP_\bB$ satisfies exponential concentration for 
Lipschitz functions in $(\cE_N, \|\cdot\|_{\textsc{mix}})$ and Proposition \ref{prop:main-lipschitz}
that $V=F(\bt;\cdot)$ is $\frac{D(R)}{\sqrt N}$-Lipschitz on $\cL=\cL_{N,R}$ for $D(R)=C_1 e^{C_1 \sqrt{R}}$, 
for some $C_1 (T, C_\ba, \bC_\star)$ for every $R \ge R_0(T,C_\ba,\bC_\star)$, all $N$, and every $F$, $\bt$ as in Theorem \ref{thm:conc}.

Further, increasing $R_0$ as needed for Lemma~\ref{lem:L-N-R-complement} and 
Proposition~\ref{prop:X-t-localization}, yields 
\begin{align*}
\sup_{(\bX_0,\bJ, \bM)\in \cL_{N,R}} |F(\bt; (\bX_0,\bJ, \bM))| \le K(R) \qquad \mbox{where}\qquad K(R): = C_\ba \max(R,e^{2 C_0 \sqrt{R} T})\,,
\end{align*} 
as well as guaranteeing that $C_2^2 := \sup_{N,\bt} \{\bE [F(\bt)^2]\}$ is finite and that 
$\bP(\cL_{N,R}^c) \le \exp(-\sqrt{RN}/C_3)$ for some $C_3(T, C_\mu, C_{\bA},C_{\bsigma})$. Plugging all this into \eqref{eq:lem25-BDG} gives us 
a family of upper bounds for $R \ge R_0$,
\[
q_N(\lambda;R) =  C e^{- \lambda\sqrt{N}/(2 D(R) C)} + e^{-\sqrt{RN}/C_3} + \frac{2}{\lambda} 
(C_2 + K(R)) e^{-\sqrt{R N}/(2C_3)}  \,.
\]
For $R=R_0$ we can embed the constant factor $2 D(R_0)$ into $C$ and further adjust $C_3$ to 
bound the pre-exponent $2 (C_2+K(R_0))$ within the factor $\exp(-\sqrt{R_0 N}/(2C_3))$ multiplying it,
resulting with $q_N(\lambda;R_0)$ as in the top line on the \abbr{rhs} of \eqref{eq:conc-fixed-bt}.
For a better tail decay, consider $R_\lambda=(\eta \log \lambda)^2 \ge R_0$,  with 
$\eta=1/(2C_1)$ so $D(R_\lambda)=  C_1 e^{C_1 \eta \log \lambda} \le C_1 \lambda/\log \lambda$
for all $\lambda \ge 4$. In addition, once $\sqrt{N}/(2C_3) \ge 4 C_0 T$ we can again embed 
the pre-exponent $2(C_2+K(R_\lambda))/\lambda$ within the factor $\exp(-\sqrt{R_\lambda N}/(2 C_3))$
multiplying it . Thus, upon adjusting the various constants we end up with 
$q_N(\lambda;R_\lambda)$ as in the bottom line on the \abbr{rhs} of \eqref{eq:conc-fixed-bt}.
\end{proof}

Setting hereafter $R$ for the larger of $R_0$ and $R_\lambda$ values from the preceding proof of Proposition \ref{prop:main-concentration}, recall that the event $\cL^c_{N,R}$ 
was already ruled out as part of the derivation of  \eqref{eq:conc-fixed-bt}. Thus, 
proceeding to prove Theorem \ref{thm:conc}, we fix $\varepsilon = N^{-k}$, $k>1$,
and apply Proposition \ref{prop:main-concentration} at the $M_N = \lceil T N^k \rceil^2$ grid 
points $\bt_{i,j}=(i \varepsilon, j \varepsilon)$ within $[0,T]^2$, to deduce by the union bound that 
 \begin{align*}
 \bP(\cL_{N,R}^c) +
 \bP \Big(\sup_{i,j} \big| F(\bt_{i,j}) - \bE [F(\bt_{i,j})]\big| > \lambda, \cL_{N,R} \Big)\leq  M_N \, q_N(\lambda)\,.
\end{align*}
It is easy to check that $2 M_N q_N(\lambda)$ is further bounded by $p_N(3 \lambda)$ of 
\eqref{eq:wts-trajectory-tail-bound} once we suitably enlarge the constant $C$ 
on the \abbr{rhs} of~\eqref{eq:wts-trajectory-tail-bound} relative to that of~\eqref{eq:conc-fixed-bt}. In addition, since the right-most term 
in \eqref{eq:lem25-BDG} exceeds one whenever $\bE[|V| {\bf 1}_{\cL_{N,R}^c}] = \bE [ |F(\bt)| {\bf 1}_{\cL_{N,R}^c}] \ge \lambda/2$, if that inequality holds for any $\bt\in [0,T]^2$, then $q_N(\lambda)$ and in turn $p_N(3\lambda)$ of \eqref{eq:wts-trajectory-tail-bound} would exceed one. Thus, we may assume \abbr{wlog} that 
\begin{equation}\label{eq:trivial-fact}
\sup_{\bt, \bs: \bt+\bs \in [0,T]^2} \{ \bE[|F(\bt+\bs)-F(\bt)| \, {\bf 1}_{\cL^c_{N,R}} ] \} \le 2 \sup_{\bt\in [0,T]^2} \bE[|F(\bt)| \one_{\cL_{N,R}^c}] \le \lambda \,.
\end{equation}
 We can then expand 
\begin{align*}
\bP \Big(\sup_{\bt \in [0,T]^2} \big| F(\bt) - \bE [F(\bt)]\big| > 3 \lambda \Big) &  \le  \bP(\cL_{N,R}^c) + \bP \Big( \sup_{i,j} |F(\bt_{i,j}) - \bE [ F(\bt_{i,j})]| > \lambda , \cL_{N,R}\Big) \\
& \quad + M_N \sup_{i,j} \mathbb P \Big(  \sup_{\bs \in [0,\varepsilon]^2} |F(\bt_{i,j}+\bs) - F(\bt_{i,j})| >\lambda, \cL_{N,R} \Big) \\ 
& \quad + \mathbf 1\Big\{\sup_{\bt,\bs} \mathbb E[ |F(\bt + \bs) - F(\bt) |\mathbf 1_{\cL_{N,R}}]>\lambda\Big\}\,.   
\end{align*}
Restricting to $\lambda > 1/\sqrt{N}$ (as otherwise $p_N(3 \lambda) \ge 1$), and using  
$p_N(3 \lambda) \gg  M_N \exp(- (\lambda^2 \wedge \lambda) N^k/C')$ (as $k>1$) with the above, 
 the stated bound of Theorem \ref{thm:conc}, 
follows from the following short-time estimates. 

\begin{lem}
There exists $C'(C_\bsigma)$, such that for every $\varepsilon \le  1$, $\lambda \ge C' \varepsilon$, 
and $F$ as in Theorem \ref{thm:conc}, 
\begin{align}\label{eq:equi-cont}
\sup_{\bt \in [0,T-\varepsilon]^2} \bP\Big(\sup_{\bs \in [0,\varepsilon]^2} |F(\bt+\bs)-F(\bt)| > \lambda, \, \cL_{N,R}\Big) 
 \le 2 e^{- (\lambda^2 \wedge \lambda) /(C' \varepsilon)} \,.
\end{align}
In particular, for any $N \ge N_0(T,C_\ba,C_\mu, \bC_\star)$ and $\lambda \ge N^{-1/2} = \varepsilon^{1/(2k)}$,
$k>1$,
\begin{align}\label{eq:equi-cont-mean}
\sup_{\bt \in [0,T-\varepsilon]^2, \bs \in [0,\varepsilon]^2} \bE \big[ |F(\bt+\bs)-F(\bt)| {\bf 1}_{\cL_{N,R}}
\big]
 \le \lambda \,.
\end{align}
\end{lem}
\begin{proof} Similarly to the computation leading to \eqref{eq:bd-F-diff-LNR}, we find that for any 
$\bt+\bs \in [0,T]^2$ and $F$ as in Theorem~\ref{thm:conc}, evaluated on 
the solution $\bX_t(\bX_0,\bJ,\bM)$ that corresponds to some $(\bX_0,\bJ,\bM) \in \cL_{N,R}$
\[
|F(\bt+\bs
)-F(\bt
)| \le \frac{2 C_\ba e^{C_0 \sqrt{R} T}}{\sqrt{N}} \,
\max_{\cY \in {\mathfrak F} } \max_{i = 1,2}  \; \{ \|\cY(t_i+s_i) -\cY(t_i)\| \} \,.
\]
When $\cY={\bf 1}$ this difference is zero, whereas in case $\cY = \bX$ and $s_i \le \varepsilon$,
assuming \abbr{wlog} that $R_0,  \bC_\star \ge 1$,  we 
have on $\cL_{N,R}$, by \eqref{eq:x-t-bound-localized} and the \abbr{rhs} of \eqref{eq:G-to-X}, that 
\begin{align}
    \| \bX_{t_i +s_i} - \bX_{t_i} \| &\leq \|\bM_{t_i + s_i} - \bM_{t_i} \| + \int_{t_i}^{t_i+ s_i} 
    \,[\,\|\bG_u\|+ C_{\bLambda} \| \bX_u\| + \sqrt{N} C_\bh \,]\, du \nonumber \\
   &\leq \|\bM_{t_i + s_i} - \bM_{t_i} \| + 3 \varepsilon \sqrt{N} \bC_\star e^{C_0 \sqrt{R} T}  \,.
   \label{eq:X-t-s}
\end{align}
Further, similarly to the \abbr{lhs} of \eqref{eq:G-Lipschitz}, on $\cL_{N,R}$,
\[
  \| \bG_{t_i +s_i} - \bG_{t_i} \| \le  \|\bJ\|_{2 \to 2} \| \, \bX_{t_i +s_i} - \bX_{t_i} \|
  \le \sqrt{R} \| \, \bX_{t_i +s_i} - \bX_{t_i} \| \,,
\]
so up to extra factor $\sqrt R$ the bound \eqref{eq:X-t-s} applies for $\cY=\bG$, and 
considering all cases we get for $\bs \in [0,\varepsilon]^2$, 
\begin{equation}\label{eq:F-t-s-temp}
|F(\bt+\bs)-F(\bt)| \le 
\frac{2 C_\ba \sqrt{R} e^{C_0 \sqrt{R} T}}{\sqrt{N}} \max_{i=1,2} \|\bM_{t_i + s_i} - \bM_{t_i} \|
+6 \varepsilon C_\ba \bC_\star \sqrt{R} e^{2 C_0 \sqrt{R} T} \,.
\end{equation} 
For some $C'>0$, when $R=R_0$ and $\lambda \ge C' \varepsilon$, the right most term in 
\eqref{eq:F-t-s-temp} can not exceed $\lambda/2$. The same 
applies for $R=R_\lambda = (\eta \log \lambda)^2$ provided $\eta \le 1/(3C_0 T)$. By the same 
reasoning, for such $\eta$ and some $C_4(T, C_\ba, R_0)>0$, the factor multiplying 
$\|\bM_{t_i + s_i} - \bM_{t_i} \|$ in \eqref{eq:F-t-s-temp}, is in both cases at most $(\sqrt{\lambda} \vee 1)/(2 C_4 \sqrt{N})$. Recall from~\eqref{eq:bd-M-sup} and the stationarity of Brownian increments,
that there exists $C(C_\bsigma)$ such that for every $L \ge \varepsilon^2 L_0(C_\bsigma)$, every $N$,
\begin{equation}\label{eq:bd-M-diff-sup}
\sup_{t\in [0,T-\varepsilon]} \bP_{\bB} \Big(\sup_{s \in [0,\varepsilon]}  \{ \|\bM_{t+s} - \bM_t\| \} > L \sqrt{N}\Big) 
\le e^{- 3 L^2/(C \varepsilon)} \,.
\end{equation}
Combining \eqref{eq:F-t-s-temp} and \eqref{eq:bd-M-diff-sup}, we thus get that for 
some $C'(C_\bsigma)$, for every $\lambda \ge C' \varepsilon$, and every $N$, $\bt=(t_1,t_2)$,
\begin{align*}
\bP\Big(\sup_{\bs \in [0,\varepsilon]^2} |F(\bt+\bs)-F(\bt)| > \lambda, \, \cL_{N,R}\Big)  & \le 2 \max_{i = 1,2} 
\bP\Big(\sup_{s \in [0,\varepsilon]}  \|\bM_{t_i+s} - \bM_{t_i}\| > 
C_4 (\lambda \wedge \sqrt{\lambda}) \sqrt{N}\Big) \\
&  \le 2 e^{- (\lambda^2 \wedge \lambda) /(C' \varepsilon)} \,,
\end{align*}
as claimed in \eqref{eq:equi-cont}. 
Next, by Cauchy-Schwarz, \eqref{eq:bd-F-sup} and \eqref{eq:equi-cont},
there exists $C(T, C_\ba, C_\mu, \bC_\star)$ such that for every $N \ge N_0(T, C_\ba, C_\mu, \bC_\star)$, every $\lambda \ge 2 C' \varepsilon$, every $\bt$, $\bs$ and all $F$,
\begin{align*}
\bE \big[ |F(\bt+\bs)-F(\bt)| {\bf 1}_{\cL_{N,R}}
\big] & \le \frac{\lambda}{2} + 2 \, \bP\big(|F(\bt+\bs)-F(\bt)| > \frac{\lambda}{2}, \, \cL_{N,R}\big)^{1/2} 
\sup_{\bt \in [0,T]^2} \{ \sqrt{\bE[F(\bt)^2]} \} \\
& \le \frac{\lambda}{2} + C e^{- (\lambda^2 \wedge \lambda) /(4 C' \varepsilon)} \,.
\end{align*}
Our assumption that $\lambda \ge \varepsilon^{1/(2k)}$ for some $k>1$ guarantees that the
right most term is at most $\lambda/2$ (as soon as $N \ge N_0$), thereby establishing 
\eqref{eq:equi-cont-mean}.
\end{proof}

\bibliographystyle{abbrv}
\bibliography{universality-diffusions}

\end{document}